\documentclass{article}
\usepackage{graphicx} % Required for inserting images
\usepackage{amsmath}
\usepackage{amsthm}
\usepackage{ amssymb }
\usepackage[utf8]{inputenc}
\usepackage{amsfonts}
\usepackage{placeins}
\usepackage{tikz}
\usepackage{tikz-cd}
\usetikzlibrary{tqft}
\usepackage{subcaption}
\usepackage{amsfonts,graphicx}
\usepackage{caption}
\usepackage{thmtools}
\usepackage{mathtools}
\usepackage{physics}
\usepackage{ textcomp }
\usepackage{bbold}

\usepackage[shortlabels]{enumitem}
\usepackage[maxbibnames=9]{biblatex}

\usepackage{tocloft}
\usepackage{hyperref}
\hypersetup{colorlinks,allcolors=black}
\newcommand{\Sp}{\mathbb{S}}
\newcommand{\To}{\mathbb{T}}

\newtheorem{theo}{Theorem}[section]
\newtheorem{defi}[theo]{Definition}
\newtheorem{lemm}[theo]{Lemma}
\newtheorem{pro}[theo]{Proposition}
\newtheorem{exa}[theo]{Example}
\newtheorem{rem}[theo]{Remark}
\newtheorem{col}[theo]{Corollary}

\numberwithin{equation}{section}
\title{ECH capacities of concave singular toric domains}
\author{Jonathan Trejos}
\date{}

\begin{document}

\maketitle
\begin{abstract}
    By definition, a toric domain has a boundary contact manifold diffeomorphic to a three dimensional sphere. In the present work we extend the definition of the toric domains in dimension four so that it admits a contact manifold diffeomorphic to a lens space $L(n,1)$. We call them `singular toric domains' since they naturally have an orbifold point. We calculate the ECH capacities for a specific subfamily of these singular toric domains that generalize the concave toric domains. Interestingly, even though we are calculating the capacities of an orbifold, the result can be adapted to study embedding problems in some of the desingularizations of these spaces, for example the unitary cotangent bundle of $\mathbb{S}^2$ or the unitary cotangent bundle of $\mathbb{R}P^2$.
\end{abstract}

%% I like this counter. I rather prefer the counter with depth=2 however i think that while I am working in the paper itself it is better if I use no limit in depth.
\setcounter{tocdepth}{2}
\tableofcontents
%\addtocontents{toc}{~\hfill\textbf{Page}\par}

\section{Introduction}
Let $(X,\omega)$ be a symplectic four-manifold, possibly with boundary or corners, non-compact, maybe disconnected. Its $\text{ECH}$ capacities are a sequence of real numbers 
\begin{equation}
    0=c_0(X,\omega)\leq c_1(X,\omega)\leq c_2(X,\omega)\leq \cdots\leq \infty
\end{equation}
The $\text{ECH}$ capacities were introduced in \cite{hutchings2011quantitative}, see also \cite{hutchings2014lecture}. We give more detail about these definitions in Section \ref{ECHfoundations}.

The following are elementary properties of the ECH capacities:
\begin{enumerate}
    \item \textit{(Monotonicity)}  If there exists a symplectic embedding $(X,\omega)\xhookrightarrow{s} (X',\omega')$, then $c_k(X,\omega)\leq c_k(X',\omega')$ for all $k$.
    \item \textit{(Conformality)} If $r>0$ then
    $$c_k(X,r\omega)=r c_k(X,\omega)$$
    \item \textit{(Disjoint union)}
    $$c_k\left(\coprod_{i=i}^n (X_i,\omega_i)\right)=\max_{k_1+\cdots+k_n=k}\displaystyle\sum_{i=1}^n c_{k_i}(X_i,\omega_i)$$
    \item \textit{(Ellipsoid)} If $a,b>0$, define the ellipsoid
    $$E(a,b)=\left\{(z_1,z_2)\in \mathbb{C}^2:\displaystyle\frac{\pi |z_1|^2}{a}+\displaystyle\frac{\pi |z_2|^2}{b}\leq 1\right\}$$
    Then $c_k(E(a,b))=N(a,b)_k$, where $N(a,b)$ denotes the sequence of all nonnegative integer linear combinations of $a$ and $b$, arranged in nondecreasing order, indexed starting at $k=0$.
\end{enumerate}

A proof of these properties can be found in \cite{hutchings2014lecture}.

The computation of the $\text{ECH}$ capacities is not an easy task but several result have been obtained. An interesting family of symplectic four-manifolds for which in several conditions it has been posible to compute the ECH capacities are the \textit{toric domains}. If $\Omega$ is a domain in the first quadrant of the plane, define the toric domain $X_\Omega$ as the set 
$$X_\Omega=\{z\in \mathbb{C}^2:\pi (|z_1|^2,|z_2|^2)\in \Omega\}$$
This sets is naturally a symplectic manifold since it is an open set in $\mathbb{C}^2$. For example, if $\Omega$ the triangle with vertices $(0,0), (a,0)$ and $(0,b)$, then $X_\Omega$ is the ellipsoid $E(a,b)$. 

The $\text{ECH}$ capacities of toric domains $X_\Omega$, when $\Omega$ is convex and does not touch the axes were computed in \cite{hutchings2011quantitative} theorem 1.11. The cases in which the region $\Omega$ touch the axis have received considerably more attention and they have special names:

\begin{defi}
    A \textit{convex toric domain} is a toric domain $X_\Omega$, where $\Omega$ is a closed region in the first quadrant bounded by the axes and a convex curve from $(a,0)$ to $(0,b)$, for $a$ and $b$ positive numbers. Similarly a concave domain $X_\Omega$ is a toric domain where $\Omega$ is a closed region in the first quadrant bounded by the axes and a concave curve from $(a,0)$ to $(0,b)$, for $a$ and $b$ positive number. 
\end{defi}

The $\text{ECH}$ capacties of concave domains were calculated \cite{Choi_2014}, as well as the capacities of the convex domains \cite{cristofaro2019symplectic}. In the present work we aim to generalize the notion of concave toric domains to consider symplectic manifolds with boundary a contact manifold diffeomorphic to a lens space. For these new concave domains, we compute combinatorial expressions to their $\text{ECH}$ capacities (see Theorem \ref{thm:concavecapacities}). To properly define this generalization of concave toric domains we define (Section \ref{symplectictoricorb}) an orbifold with one singularity point that plays the role of $\mathbb{C}^2$ in the definition of concave domains. This singularity point can be removed in several ways. One interesting way to remove the singularity point is using the techniques for almost toric fibration introduced by Symington \cite{symington2002dimensions}. With the use of these techniques it is possible to recover as concave domains well known spaces as the unit cotangent bundle of $\Sp^2$ as well as the unit cotangent bundle of $\mathbb{R}P^2$ (see Examples \ref{S2} and \ref{RP2}) which recently some interesting properties where found by Ferreira and Ramos \cite{Ferreira_2022}.

%We also want this present work to ground the basis to generalize some of the results obtained for the classic concave toric domains. We beging this project by generalizing the ball packing result from \cite{Choi_2014} in Section \ref{ballpacking}.

\subsection{Symplectic Toric Orbifolds}
\label{symplectictoricorb}

Let $n$ be a positive integer, in this section we  define a symplectic orbifold $M_n$ that we use as an ambient space to our definitions of a toric domain with singularities. 

\subsubsection{Construction of the toric orbifold $M_n$.}
    On $\mathbb{R}_{> 0}^2\times \To^2$ with coordinates $(r_1,r_2,\theta_1, \theta_2)$ lets consider the standart symplectic form:
    \begin{equation}
    \label{eq:standart-form}
    \omega_{std}=\frac{1}{2\pi}\left[\dd r_1 \wedge\dd \theta_1+\dd r_2 \wedge\dd \theta_2\right]
    \end{equation}    
    Let's also consider the cone $V_{n}$ with coordinates $r_1,r_2$, define as 
    $$V_n=\{r_1(n,1)+r_2(0,1):r_1,r_2\geq 0\}$$
    We call the rays $\{r_1(n,1):r_1\geq 0\}$ and $\{r_2(0,1):r_1\geq 0\}$ the axis of $V_n$. Consider $V_n\times \To^2$ where $\To^2=(\mathbb{R}/(2\pi\mathbb{Z}))^2$ with coordinates $(\theta_1,\theta_2)$. This is a symplectic manifold with the symplectic form \ref{eq:standart-form} restricted to it . 

       We define $M_n$ as a quotient of $V_{n}\times \To^2$. Consider the equivalence relationship on $V_{n}\times \To^2$ as follows
    \begin{itemize}
    \item For $\{r_1(n,1):r_1\geq 0\}\times \To^2$ for a fixed $r_1$ we colapses the orbits $\partial_{t_1}$.
    \item For $\{r_2(0,1):r_2\geq 0\}\times \To^2$ for a fixed $y$ we colapses the orbits $-\partial_{t_1}+n\partial_{t_2}$.
\end{itemize}
    Then we define $M_n=(V_n\times \To^2)/\sim$.  Notice that we have a canonical map $\pi_n:M_n\rightarrow V_n$. 

    \begin{lemm}
    $M_n$ is a symplectic toric orbifold with a canonical projection map $\mu_n=2\pi\, \pi_n$.
\end{lemm}

\begin{proof}
    The construction of the quotient $M_n$ can be seen as an action of the group $\mathbb{Z}_n$ over $V_n\times \To^2$ that preserve the symplectic form. The only orbifold singularity is the origin. It follows that $M_n$ is a simplectic orbifold according to \cite[Sec.~3]{lerman1994symplectictoricorbifolds}.

    Let's check that $\mu_n$ is the moment map. For calculating the moment map we follows the notation of \cite[Sec.22]{CannasdaSilva2001Lectures} The relevant vector fields to consider are $X=\partial_{\theta_1}$ and $Y=-\partial_{\theta_1}+n\partial_{\theta_2}$. Notice that for this vector fields 
    \begin{align*}
        \iota_{X}(\dd r_1\wedge\dd \theta_1+\dd r_2\wedge\dd \theta_2)=-\dd r_1 \\
        \iota_{Y}(\dd r_1\wedge\dd \theta_1+\dd r_2\wedge\dd \theta_2)=\dd r_1-n\dd r_2
    \end{align*}
    Then $\mu_n^X(r_1,r_2,\theta_1,\theta_2)=(r_2,-r_1)\times (1,0)=r_1 $ and $\mu_n^X(r_1,r_2,\theta_1,\theta_2)=(r_2,-r_1)\times (-m,n)=-r_1+nr_2 $. That is 
    \begin{align*}
        -\dd r_1 = -\dd \mu^X(r_1,r_2,t_1,t_2) \\
        \dd r_1 -n\dd r_2= -\dd \mu^Y(r_1,r_2,t_1,t_2)
    \end{align*}
    From the above equation we conclude that $\mu_n$ is moment map. 
\end{proof}

    Notice that if $n=1$ then $M_1$ is symplectomorphic to $\mathbb{C}^2$.

    As it is pointed out in the above proof, $M_n$ has a unique orbifold point at the origin. 
    
\subsubsection{Toric Domains in $M_n$.}
With the definition of the ambient space $M_n$ in place we can define a toric domain in $M_n$.
\begin{defi}
    \label{domain}
    A domain $\Omega$ in $V_n$ is a bounded subset of $V_n$ for which there exists an embedded curve $a:[0,1]\rightarrow V_n$ with $a(0)=(a_0 n, 1)$ and $a(1)=(0,a_1 1)$  such that $\partial \Omega=a[0,1]\cup \{t_1(n,1):0\leq t_1\leq a_1 \}\} \cup\{t_2(0,1):0\leq t_1\leq a_2\}$. We also write $\partial^+\Omega=a[0,1]$.
\end{defi}

\begin{defi} Let $\Omega$ be a domain in $V_n$. We called the symplectic orbifold $X_\Omega:=\pi^{-1}(\Omega)$ a \textit{toric domain} in $M_n$. We say that $X_\Omega$ is a \textit{concave toric domain} if the complement of  $\Omega$ in $V_n$ is a convex set. Notice that $\partial X_\Omega=\mu_n^{-1}(\partial^+\Omega)$. In case no ambient space is mention we call $X_\Omega$ a singular toric domain. 
\end{defi}

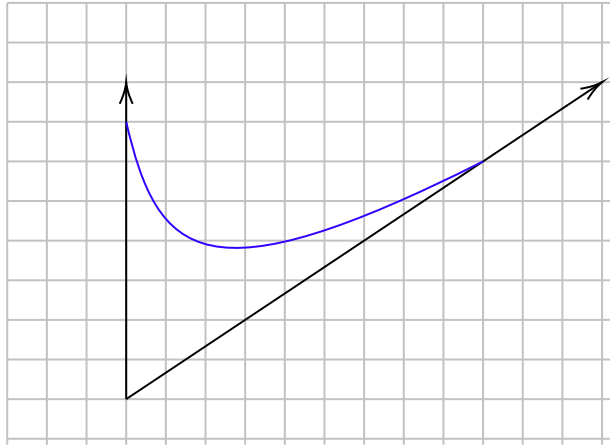
\begin{figure}
    \centering
    \tikzset{every picture/.style={line width=0.75pt}} %set default line width to 0.75pt        

\begin{tikzpicture}[x=0.75pt,y=0.75pt,yscale=-1,xscale=1]
%uncomment if require: \path (0,238); %set diagram left start at 0, and has height of 238

%Shape: Grid [id:dp18791668271167783] 
\draw  [draw opacity=0] (166,4) -- (470,4) -- (470,227) -- (166,227) -- cycle ; \draw  [color={rgb, 255:red, 194; green, 194; blue, 194 }  ,draw opacity=1 ] (166,4) -- (166,227)(186,4) -- (186,227)(206,4) -- (206,227)(226,4) -- (226,227)(246,4) -- (246,227)(266,4) -- (266,227)(286,4) -- (286,227)(306,4) -- (306,227)(326,4) -- (326,227)(346,4) -- (346,227)(366,4) -- (366,227)(386,4) -- (386,227)(406,4) -- (406,227)(426,4) -- (426,227)(446,4) -- (446,227)(466,4) -- (466,227) ; \draw  [color={rgb, 255:red, 194; green, 194; blue, 194 }  ,draw opacity=1 ] (166,4) -- (470,4)(166,24) -- (470,24)(166,44) -- (470,44)(166,64) -- (470,64)(166,84) -- (470,84)(166,104) -- (470,104)(166,124) -- (470,124)(166,144) -- (470,144)(166,164) -- (470,164)(166,184) -- (470,184)(166,204) -- (470,204)(166,224) -- (470,224) ; \draw  [color={rgb, 255:red, 194; green, 194; blue, 194 }  ,draw opacity=1 ]  ;
%Straight Lines [id:da6280534481661273] 
\draw    (226,204) -- (226,46) ;
\draw [shift={(226,44)}, rotate = 90] [color={rgb, 255:red, 0; green, 0; blue, 0 }  ][line width=0.75]    (10.93,-3.29) .. controls (6.95,-1.4) and (3.31,-0.3) .. (0,0) .. controls (3.31,0.3) and (6.95,1.4) .. (10.93,3.29)   ;
%Straight Lines [id:da5778110298984791] 
\draw    (226,204) -- (464.34,45.11) ;
\draw [shift={(466,44)}, rotate = 146.31] [color={rgb, 255:red, 0; green, 0; blue, 0 }  ][line width=0.75]    (10.93,-3.29) .. controls (6.95,-1.4) and (3.31,-0.3) .. (0,0) .. controls (3.31,0.3) and (6.95,1.4) .. (10.93,3.29)   ;
%Curve Lines [id:da9504749243436394] 
\draw [color={rgb, 255:red, 50; green, 0; blue, 255 }  ,draw opacity=1 ]   (226,64) .. controls (242,133) and (265,155) .. (406,84) ;

\end{tikzpicture}
    \caption{Example of a Toric Concave Domain}
    \label{Toric Concave Domain}
\end{figure}

\begin{exa}[Singular Ellipsoids]  
    Let $a$ and $b$ be real positive numbers. Consider the domain $\Omega$ in $V_{n}$ defined as the convex hull of the vertices $(0,0), a(n,1)$ and $b(0,1)$. We denote $X_{\Omega}$ by $E_{n}(a,b)$ and we call it the \textit{ellipsoid with singularities} of periods $a$ and $b$. We define the ball with singularities as $B_{n}(a)\vcentcolon=E_{n}(a,b)$. 
\end{exa}

We can also consider examples that are not bounded.

\begin{exa} Given $a>0$ we define \textit{singular cylinders} as
\begin{align*}
    Z_n^L(a):=\{t_1(n,1)+t_2(0,1):t_1\leq a\},&& Z_n^R(a):=\{t_1(n,1)+t_2(0,1):t_2\leq a\}
\end{align*}
It is easy to see that $Z^L(a)$ and $Z^R(a)$ are symplectomorphic. We also define the \textit{singular non-disjoint union} of singular cylinders. 
$$Z_n(a,b):= Z^U(a) \cup Z^D(b) $$
For $a,b>0$.
\end{exa}

\begin{exa}
    Since $M_1$ is symplectomorphic to $\mathbb{C}^2$ a toric domain $X_\Omega$ in $M_1$ is just a standart toric domain. 
\end{exa}

\subsubsection{The boundary of a toric domains in  $M_n$.}
\label{subsec: ToricLensspaces}
Lets begin by formulating a definition of the lens space in a way that it is useful for our pourposes. 

\begin{defi}
\label{defi:lensspace}
    Consider $I\times \To^2$ with coordinates $x,t_1,t_2$ and oriented by $\{\partial_x, \partial_{t_1}, \partial_{t_2}\}$. We define the lens space $L(n,1)$ as the quotient $I\times \To^2/\sim$ where $\sim$ is colapsing the integral curves of the field $\partial_{t_1}$ in $\{0\}\times \To^2$ and the integral orbits of the field $-\partial_{t_1}+n\partial_{t_2}$ in $\{1\}\times \To^2$.
\end{defi}

This definition coincide with the given in \cite[Sec. 9]{symington2002dimensions}.

\begin{lemm}
\label{lem:exilens}
Let $a:[0,1]\rightarrow \bar{V}_n$ be a embedding such that $a(0)$ lies in the ray $\{t(n,1):t>0\}$ and $a(1)$ lies in the ray $\{t(0,1):t>0\}$ then
\begin{enumerate}[i]
    \item $Y_a\vcentcolon=\mu_n^{-1}(a([0,1]))$ is diffeomorphic to the lens space $L(n,1)$.
    \item Suppose that $a\times a'>0$ then $Y_a$ is a contact manifold with a $\To^2$-invariant contact form
    \begin{equation}
    \label{eq:contacsstructure}
    \lambda_a=a_1 \dd t_1+a_2 \dd t_2     
    \end{equation}
    where $a=(a_1,a_2)$ and $a\times a'=a_1a_2'-a_1 a_2'$.
\end{enumerate}
\end{lemm}

\begin{proof}

$i.$ Notice that $\mu_n^{-1}(a([0,1]))$ is exactly the quotient described in the definition \ref{defi:lensspace}. 

$ii.$ Now consider the Liouville vector field
$$V=r_1\partial_{r_1}+r_2\partial_{r_2}$$
of the symplectic $2$-form \eqref{eq:standart-form}. The conditions over the function $a$ ensure us that $V$ and $Y_a$ are transversal. Then $\iota_V\omega$ is a contact structure equal to 
$$r_2\dd t_1+r_1 \dd t_2$$
Replacing $t_1$ and $t_2$ by $(a_1,a_2)$ give us the equation \eqref{eq:contacsstructure}. It is easy to check that this form is $\To^2$ invariant.
\end{proof}

\begin{defi}
A lens space $Y_a$ as in Lemma \ref{lem:exilens} is call \textit{toric lens space}. Futhermore, we say that $Y_a$ is a \textit{concave toric lens space} if $a'\times a''<0$
\end{defi}

We can easily see that the Reeb vector field $R_a$ and the contact structure $\xi_a$ of a toric lens space $Y_a$ are given by:
\begin{align}
\label{eq:Reebandstructure}
R=\displaystyle\frac{1}{a\times a'} \left( a'_2\partial_{t_1}-a'_1\partial_{t_2}\right), && \xi_a=\text{span}\{\partial_x, - a_2\partial_{t_1}+a_1\partial_{t_1}\}
\end{align}

Notice that if $X_\Omega$ is a toric domain in $M_n$. The boundary $\partial X_\Omega$ is a toric lens space diffeomorphic to $L(n,1)$. Futhermore, if $X_\Omega$ is concave then $\partial X_\Omega$ is a concave toric lens space. We can also write $\partial X_\Omega=Y_a$ where $a:[0,1]\rightarrow \mathbb{R}^2$ is the function describing $\partial^+ \Omega$.
\subsection{ECH Capacities of toric domains in $M_n$}

In this section we give a description of the ECH capacities of concave toric domains in $M_n$. We begin by describing the capacities of ellipsoids with singularities. We give the proofs of the results of this section in Section \ref{ECH-of-Lens-Spaces}.

\subsubsection{ECH capacities of Ellipsoids with singularities.}

Given two real positive numbers $a,b$ and two positive integer numbers $n,m$, we define the sequence $N^n(a,b)$ as the sequence of numbers of the form $a k_1+bk_2$ such that there exist an integer $l$ such that $k_1+ k_2=ln$, the sequence $N^n(a,b)$ is organized by increasing order with repetitions. We denote the $k$-th number of the sequence $N^n(a,b)$ by $N_k^n(a,b)$. So the ECH capacities of ellipsoids with singularity are given by the following Lemma:

\begin{lemm}
\label{lem:capacitiesellipsoids}
     The ECH capacities of an ellipsoid with singularities are given by the sequence defined above, i.e,

\begin{align}
\label{capLmn}
c_k(E_n(a,b))=N_k^n(a,b)    
\end{align}

where $a,b$ are positive real numbers and $n$ are positive integer numbers. 
\end{lemm}

We give a proof of Lemma \ref{lem:capacitiesellipsoids} in Section \ref{ECH-of-Lens-Spaces}.

The above calculations can be used to recover some results in \cite{Ferreira_2022}, we discuss this in Section \ref{sec:desingularization}.

\subsubsection{ECH capacities of concave toric domains on $M_n$}

One of our main result is a combinatorial description of the ECH capacities of concave toric domains on $M_n$. The precise statement is Theorem \ref{thm:concavecapacities} this stamente is similar to the one given in \cite[Sec.~ 1.6]{Choi_2014}.

\begin{defi}
    \label{integral-paths}
    An \textit{integral path} in $V_n$ is a piecewise linear continuos path $\Lambda$  with starting point a lattice point in the line $\{t(n,1):t\geq 0\}$ and end point at a lattice point in the $y$-axis. We say the integral path $\Lambda$ is \textit{concave} if it lies above any of the tangent lines at its smooth points. 
\end{defi}

\begin{defi}
\label{Counting points}
    If $\Lambda$ is a concave integral path in $V_n$, define $\mathcal{L}_{n}(\Lambda)$ to be the number of lattice points in the region bounded by $\Lambda$ and the axis of $V_n$. Without counting the lattice points in $\Lambda$.
\end{defi}

\begin{defi}
\label{def:omegalength}
        Let $X_\Omega$ be the concave toric domain in $M_n$. Suppose that $\Lambda$ is a concave integral path in $V_n$, define the $\Omega$-\textbf{length} of $\Lambda$, as follows. For each edge $v$ let $p_v$ be a point in $\partial^+\Omega$ such that $\Omega$ is contained in the closed half-plane above the line through $p_v$ parallel to $v$. Then 
    $$l_{\Omega}(\Lambda)=\displaystyle\sum_{v\in \textit{Edges}(\Lambda)} v\times p_v$$
    Here $\times$ denote the cross product. Note that if $p_v$ is not unique then the value $v\times p_v$ does not depend on the choice of $p_v$. 
\end{defi}

\begin{exa}
    \label{exa:ellipsoidaction}
    Consider $X_\Omega=E_n(a,b)$ and  let $\Lambda$ a concave integral path. Let $L$ be the line that is tangent to $\Lambda$ at $(x_0,y_0)$ and leaves $\Lambda$ in the unbounded part $V_n$. Then  
    $$l_\Omega(\Lambda)=a t_1+b t_2$$
    where $t_1$ is the horizontal distant from the $y$-axis to the point $(x_0,y_0)$ and $t_2$ is the horizontal distant from $(x_0,y_0)$ and the $(n,1)$-axis. 
\end{exa}

\begin{theo}
    \label{thm:concavecapacities}
    If $X_\Omega$ is any concave toric domain of $M_n$, then its $ECH$ capacities are given by 
    $$c_k(X_\Omega)=\max\{l_\Omega(\Lambda):\mathcal{L}_n(\Lambda)=k\} $$
Here the maximum is over concave integral paths $\Lambda$ on $V_n$.
\end{theo}

The proof of Theorem \ref{thm:concavecapacities} is given at the end of Section \ref{ECH-of-Lens-Spaces}. We call the right hand side of the equation \textit{the combintorial expression of the ECH capacities for concave toric domains.} In case there is no risk of ambiguity we will call it simple the combinatorial expression of the ECH capacities.

\subsection{Packings and ECH capacities}
\label{ballpacking}

In this section we want to describe an specific packing for the concave toric domains in $X_\Omega$, it turns out that for this specific packing, the ECH capacities of $X_\Omega$ and the packing coincide. 

\subsubsection{Weight Expansions}
\label{weightexpansions}

We describe the packing on family of concave toric domains that we call rational:

\begin{defi} Let $X_\Omega$ be a  concave toric domain in $M_n$. As in Definition \ref{domain} let $a:[0,1]\rightarrow V_n$ be the curve such that $\partial^+\Omega=a[0,1]$ curve such that  $a(0)=(a_0 n, a_0 )$, $a(1)=(0,a_1 1)$ and  $\partial \Omega=a[0,1]\cup \{t_1(n,1):0\leq t_1\leq a_1 \}\} \cup\{t_2(0,1):0\leq t_1\leq a_2\}$. We say that $X_\Omega$ is a \textit{rational toric domain} if $a$ is piecewise linear and $a'$ is rational whenever is defined.     
\end{defi}

Suppose that $X_\Omega$ is a rational concave toric domain in $M_n$. The \textit{weight expansion} of $\Omega$ is a finite unordered list of (possibly repeated) positive real numbers $w(\Omega)=\{a,a_1,\dots, a_n\}$ analogous to the weight expansion of a concave toric domain define in \cite[Sec.~ 1.3]{Choi_2014}. 

We define the weight expansion of a rational concave toric domain $X_\Omega$ in $M_n$ inductively as follows. If $\Omega$ is the triangle with vertices $(0,0),(na,a)$ and $(0,a_0)$ then $w(\Omega)=(a_0)$.  

Otherwise let $a>0$ be the largest real number such that the triangle $(0,0),(na,a)$ and $(0,a)$ is contained in $\Omega$. Call this triangle $\Omega_1$. The line $y=a$ intersect $\partial^+ \Omega$ in a line segment from $(x_2,a)$ to $(x_3,a)$ with $x_2\leq x_3$. Let $\Omega'_2$ denote the portion of $\Omega$ above the line $y=a$ and to the left of $x=x_2$. By applying the translation $(x,y)\rightarrow (x,y-a)$ to $\Omega'_2$ we obtain a $\Omega_2$ in $\mathbb{R}^2$. Let $\Omega'_3$ denote the portion of $\Omega$ above the line $y=a$ and to the right of the line $x=x_3$. By first applying the translation $(x,y)\rightarrow (x-a,y-a)$ and then multiplying by $\begin{pmatrix} 0 & 1 \\ -1 & n \end{pmatrix}\in \text{SL}_2(\mathbb{Z})$ we obtain another domain  $\Omega_3$ in $\mathbb{R}^2$. Lets consider the concave toric domains $X_{\Omega_2}$ and $X_{\Omega_3}$. We want to define the weight expansion of $\Omega$ as 
\begin{equation}
    \label{eq:weightexpantion}
w(\Omega)=w(\Omega_1)\cup w(\Omega_2)\cup w(\Omega_3)    
\end{equation}

We interpret the union $\cup$ as an union with repetitions. The sets $w(\Omega_2)$ and $w(\Omega_3)$ are the weight expansions defined for concave toric domains in $\mathbb{R}^4$ defined in \cite[Sec.~ 1.3]{Choi_2014}. We explain this sequence below. 

For simplicity, we only explain how to calculate the weight expansion of $X_{\Omega_2}$ becasue the definition for $X_{\Omega_1}$ is the same. If $X_{\Omega}=B(a_2)$ is a ball then $w(X_\Omega)=a_2$. Otherwise, let $a>0$ be the largest real number such that the triangle with vertices $(0,0),(a,0)$ and $(0,a)$ is contained in $\Omega$. Let $\Omega'_2$ be the triangle with vertices $(0,0),(a,0)$ and $(0,a)$. The line $x+y=a$ intersect $\partial \Omega_2$ in a line segments from $(x_4,a-x_4)$ to $(x_5,a-x_5)$ with $x_4\leq x_5$. Let $\Omega'_4$ be the portions of $\Omega_2$ that is above the line $x+y=a$ and to the left. Similarly, let $\Omega'_5$ be the portions of $\Omega_2$ that is above the line $x+y=a$ and to the right. By applying the translation $(x,y)\rightarrow (x,y-a)$ to $\Omega'_3$ and then multiplying by $\begin{pmatrix} 1 & 0 \\ 1 & 1 \end{pmatrix}$ we find a new domain $\Omega_4$. By first applying the translation $(x,y)\rightarrow (x-a,y)$ and then multiplying by $\begin{pmatrix} 1 & 1 \\ 0 & 1 \end{pmatrix}\in \text{SL}_2(\mathbb{Z})$ we obtain a new region $\Omega_5$. We now define 
$$w(\Omega_2)=w(\Omega'_2)\cup w(\Omega_4)\cup w(\Omega_5)$$
This gives a complete description of how to compute Equation \ref{eq:weightexpantion}.

We usually write the weight expansion $w(X_\Omega)$ of a concave toric domain in $M_n$ as a sequence $w(X_\Omega)=(a_0,a_1,a_2,\dots)$ in a non-decreasing way.

\begin{rem}
    If $X_\Omega$ is a rational concave toric domain, then the weight expansion is finite this leads to some small simplifications. However, this condition is not necessary if we admit a infinite weight sequence and use some limiting arguments. 
\end{rem}

\subsubsection{Packing Theorem}
Now we have the elements to describe the packing theorem.

\begin{theo}
\label{the:ballpacking}
Let $X_\Omega$ be a rational concave toric domain in $M_n$ with weight expansion $w(\Omega)=(a_0,a_1\dots,a_s)$ then 
$$c_k(X_\Omega)=c_k\left(B_{n}(a_0)\amalg \coprod_{k=1}^s B(a_k)\right)$$ 
\end{theo}

Notice that by taking $n=1$ we recover \cite[Theo.~1.4]{Choi_2014}. Which is the standard ball packing theorem for concave domains.

%% As I understand it is important to exposed as the main result the sort of result that is going to called the attention. 
\subsection{Desingularization of Singular Toric Domains}
\label{sec:desingularization}
In this section we explain an useful perspective related to the study of toric domains in $M_n$. As it was mention before, the manifold $M_n$ have a orbifold singularity in the origin. In the case we want to study obstructions for symplectic embeddings in $M_n$ using the ECH capacities describe by Theorem \ref{thm:concavecapacities} we might run into a technical problem. More specifically, suppose that $X_\Omega$ and $X_\Omega'$ are concave toric domain in $M_n$, suppose aditionally that $X_\Omega\hookrightarrow X_\Omega'$ symplectically, if there is not other condition over the embeddings it does not follows directly that $c_k(X_\Omega)\leq c_k(X_\Omega')$. This is due to the singularity at the origin. The monotonicity of the ECH capacities follows from the existence of the ECH cobordism map when the cobordism $X_\Omega'/X_\Omega$ is a weakly exact Liouville domain, in particular, a smooth cobordism, as it is explained in the proof of \cite[Lem.~1.14]{hutchings2014lecture}. Thefore, it is unknown if $c_k(X_\Omega)\leq c_k(X_\Omega')$ follows from $X_\Omega\xhookrightarrow{s} X_\Omega'$ in general for this kind of toric domains in general. 

There are several ways to deal with this technical problem. One of them is to assume that the symplectic embedding $X_\Omega\hookrightarrow X_\Omega'$ send the orbifold singularity to the orbifold singularity. Another interesting approach is to consider a desingularizations of the space $M_n$ in which Theorem \ref{thm:concavecapacities} and Theorem \ref{the:ballpacking} remains true or it is easy to modified. In this section we explore briefly two possible desingularizations, the first by \textit{rational blow up} and the second by using \textit{almost toric fibrations}.

\subsubsection{Rational blow up}
Lets consider our ambient space $M_n$ and a positive real number $\delta$. By removing the interior of the singular ball $B_n(\delta)$ and collapsing the orbits in the direction $\partial_{t_1}$ over $\partial B_n(\delta)$ we obtain a smooth symplectic manifold $M_n^\delta$. This process is call rational blow up and it is a particular case of a symplectic cut \cite{lerman1995symplectic}. A new moment map $\mu_n^\delta:M_n^\delta\rightarrow V_n^\delta$ where $V_n^\delta=\{r_1(n,1)+r_2(0,1)\in V_n:r_2\geq\delta\}$. 

Let $X_\Omega$ be a concave toric domain in $M_n$. Lets also suppose that after the rational blown-up the boundary $\partial X_\Omega$ in $M_n^\delta$ is not affected. Then after the rational blow up $X_\Omega$ transform into a new symplectic manifold $X_\Omega^\delta$. We would like to find a version of Theorem \ref{thm:concavecapacities} but for $X_\Omega^\delta$ in this new setting. Let denote $Y_a=\partial X_\Omega$ for some function $a$. Write $a^\delta=a-\delta \,(0,1)$ and notice that $\partial X^\delta_\Omega$ is a contact manifolds $Y^\delta_a:=\partial X^\delta_\Omega$. Notice that the dynamics of $Y_a$ and $Y^\delta_a$ are the same but the action of the orbit change. 

\begin{defi}
    For a concave polygonal path $\Lambda$ we define the \textit{a $\Omega^\delta-$lenght} as 
    $$l_\Omega^\delta(\Lambda)=l_\Omega(\Lambda)-\delta y(\Lambda)$$
    Where $y(\Lambda)$ is the horizontal displacement of $\Lambda$.
\end{defi}

The version of the Theorem \ref{thm:concavecapacities} for $X^\delta_\Omega$ is as follows:

\begin{theo}
    \label{thm:blowupconcavecapacities}
    Let $X_\Omega$ be a concave toric domain in $M_n$ and $\delta>0$. Suppose that $M_n^\delta$ contains completely $\partial X_\Omega$. Then the ECH capacities of $X^\delta_\Omega$ are given by 
    $$c_k(X^\delta_\Omega)=\max\{l^\delta_\Omega(\Lambda):\mathcal{L}_n(\Lambda)=k\} $$
    Here the maximum is over concave integral paths $\Lambda$ on $V_n$.
\end{theo}

We give the proof of this version of the theorem at the end of Section \ref{sec:ECHspectrum}.

\begin{rem}
    Theorem \ref{thm:blowupconcavecapacities} can be modified in similar ways when we consider several blow-ups of $M_n$.
\end{rem}

\subsubsection{Almost Toric Fibration}
In this section we mention how the almost toric fibration is related to the present work. The almost toric fibrations were introduced by Symington \cite{symington2002dimensions} and a nice introduction to the topic can be found in \cite{evans2022lectureslagrangiantorusfibrations}.

In our four dimensional setting recall that a toric manifold $M$ is a symplectic manifold equippend with an effective Hamiltonian  action of $\To^2$. Equivalently, a toric symplectic manifold $M$  is a completely integrable system with elliptic singularities.  An \textit{almost toric fibration} or ATF is a completely integrable system on a four-manifold $M$ with elliptic singularites and focus-focus singularities.

%When we consider $M$ to be a toric symplectic manifold. Delzant stablish a one-to-one correspondence between compact toric symplecitc manifolds (up to equivariant symplectomorphism) and Delzant poligon (up to $\text{AGL}_2(\mathbb{Z})$) \cite[Sec. ~28.1]{CannasdaSilva2001Lectures}. Also If $\Delta$ is the Delzant poligon associated to $M$ then there is a map $\mu:M\rightarrow\Delta$ which stablish a lagrangian fibration.

%For a four dimensional toric symplectic manifold $M$, we can identify the singular points of the Hamiltonians in terms of the moment map image as folows,  the pre-image of each vertex in the moment is a single point for which the moment map has an elliptic singularity of corank two. The pre-image of a point on the interior of an edge is diffeomorphic to a circle, each of these points correspond to an alliptic singularity of corank one. Finally, if we take a point in the interior of the polygon the pre-image correspond to a two dimensional torus and each of its points is a regular point. 

An interesting property of the almost toric fibrations is that some of the orbifolds points of a toric simplecitc manifold can be desingularized transforming the toric simpletic manifold into a almost toric fibration. The most important argument for our present work is that for $M_2$ and $M_4$ the orbifolds singularities are $A_n$ singularities that admits desingularizations using almost toric fibrations \cite[Sec. 7.3]{evans2022lectureslagrangiantorusfibrations}.  Futhermore, the desingularizations of $M_2$ is symplectomorphic to $T^*\mathbb{S}^2$ and $M_4$ is symplectormophic to $T^*\mathbb{RP}^2$ in way that the toric map is preserved. This in turn implies that symplectic embeddings of $T^*\mathbb{S}^2$ and   $T^*\mathbb{RP}^2$ can be study from the perspective of the spaces $M_2$ and $M_4$.

\subsection{Some applications}

In this section we give some applications of the principal results. 

\subsubsection{Singular ball capacity.}

Let $X_\Omega$ be a four dimensional toric domain in $M_n$ symplectic manifold. We define \textit{$n$-singular ball} capacity as
\begin{equation}
    c_{B_{n}}(X)=\sup\{r:B_{n}(r)\hookrightarrow X\}
\end{equation}
where we consider the embedding to be symplectic and to send the orbifold singularity into the orbifold singularity. In this section we want to stablish the following proposition:
\begin{pro}
    \label{pro:calballcapacity}
    Let $X_\Omega\subset M_n$ be a concave toric domain with the first term of the weight expansion (see section \ref{weightexpansions}) equal to $a_1$. Then 
    $$c_{B_n}(X_{\Omega})=a_1$$
    That is the embedding given by the inclusion is optimal. 
\end{pro}

To prove this we need to stablish some straighforward but important observations. Recall that the capacities of $B_n(a)$ are given by Lemma \eqref{lem:capacitiesellipsoids}. This equation can be rewritten as 
\begin{equation}
    c_k(B_n(a))=and
\end{equation}
Where $d$ is the positive integer such that 
\begin{equation}
    \displaystyle\frac{d^2p-d(p-2)}{2}\leq k \leq \displaystyle\frac{d^2p+d(p+2)}{2}
\end{equation}
From the property of the disjoint union of the ECH capacities, we also have
\begin{equation}
    \label{eq:capdisjointpack}
    \begin{split}c_k\left(B_n(a_1)\amalg\coprod_{i=2}^k B(a_i)\right)&=\max\{ a_1nd_1+\sum_{i=2}^na_id_i: \\
    &\frac{d_1^2n-d_1(n-2)}{2}+\sum_{i=2}^k\frac{d_i^2+d_i}{2}\leq k\}
    \end{split}
\end{equation}

\begin{proof}[proof of Proposition \ref{pro:calballcapacity}]
    By definition we have that $a \leq c_{B_n}(X_\Omega)$. By Equation \ref{eq:capdisjointpack} we have that 
    $$c_1(X_\Omega)=\max\{na_1,\,a_2\dots\,a_k\}$$
    From this expression, it is easy to check that indeed $c_1(X_\Omega)=na_1$. Which implies that $c_{B_n}(X_\Omega)\leq a$.
\end{proof}

\subsubsection{Almost toric fibrations.}
With the use of Lemma \ref{lem:capacitiesellipsoids} and techniques from almost toric fibration we can recover some results for the disk bundle of $\mathbb{S}^2$ and the disk bundle of $\mathbb{R}P^2$ discussed in \cite{Ferreira_2022}.
\begin{exa}
\label{S2}
In the Lemma \ref{lem:capacitiesellipsoids} take $n=2$ and $a=b=1$ then $$c_k(B_2(1))=N^2_k(1,1)=(0,2,2,2,4,4,4,4,4,6,6,6,6,6,6,6, \dots)$$
Notice that $2\pi c(B_2(1))$ are exactly the capacities for $D^*\Sp^2$ calculated in \cite{Ferreira_2022}.  Using the observations of \cite{wu2015exotic} in Section 3. It can be shown that after a desingularization of the orbifold point $2\pi c(B_2(1))$ we obtained an almost toric fibration symplectomorphic to $D^*\Sp^2$. This recovers part $(i)$ of Theorem 1.3 of \cite{Ferreira_2022}. 
We can check that from the moment map it follows $\text{int } B(1)\xhookrightarrow{s} B_2(1)$ and by the ECH capacities this embedding is sharp. We conclude that $c_{Gr}(D^*\Sp^2)=2\pi$ from which recover part (i) of \cite[Theo.~1.1]{Ferreira_2022} follows. 
\end{exa}
\begin{exa}
\label{RP2}
Similarly take $n=4$ and $a=b=1$ in Lemma \ref{lem:capacitiesellipsoids} then
    $$c_k(B_4(1))=N^4_k(1,1)=(0,4,4,4,4,4,8,8,8,8,8,8,8,8,8, \dots)$$
and notice that $\pi c(B_4(1))$ are exactly the capacities of $D^*\mathbb{R}P^2$. As in the previous case we can use the observations of \cite{wu2015exotic} in Section 3, It can be shown that after desingularization of the orbifold point $\pi B_4(1)$  can  rational blown down trade is symplectomorphic $D^*\mathbb{R}P^2$. This recovers part (ii) of Theorem 1.3 of \cite{Ferreira_2022}. Using techniques from almost toric fibration as in  \cite{casals2022full} we can see that $B(1)\xhookrightarrow{s} B_4(1)$ from which part (ii) of \cite[Theo~1.1]{Ferreira_2022} follows. In particular, $c_{Gr}(D^*\mathbb{R}P^2)=2\pi$.
\end{exa}
\subsection{The rest of the paper.}
\label{sec:restofthepaper}

Now lets give an idea of the proof of Theorem \ref{thm:concavecapacities} and Theorem \ref{thm:concavecapacities}. Lets begin by the following simple result.
\begin{lemm}
    Suppose that $X_\Omega$ is a concave toric domain in $M_n$ with weight expansion $\textbf{w}=(w_0;w_1,\dots, w_s)$. Then 
    \begin{equation}
        B_n(w_0)\sqcup\bigsqcup_{i=1}^s B(w_i)\xhookrightarrow{s} X_\Omega
    \end{equation}
\end{lemm}
    
\begin{proof}
    This is similar to \cite[Lem.~  1.7]{Choi_2014}. The only difference is that the singular ball $B_n(w_0)$ doest not require the use of the Traynor trick since it is embedded by inclusion. 
\end{proof}

It follows from the lemma above and the monotonicity of ECH capacities that 
%\section{Main Theorem}
\begin{equation}
    \label{eq:ballpack<capacities}
    c_k\left(B_n(w_0)\sqcup\bigsqcup_{i=1}^s B(w_i)\right)\leq c_k(X_\Omega)
\end{equation}

Now, in Section \ref{sec:com<packing} we prove that 
\begin{equation*}
\max\{l_\Omega(\Lambda):\mathcal{L}_n(\Lambda)=k\} \leq c_k\left(B_n(w_0)\sqcup\bigsqcup_{i=1}^s B(w_i)\right)
\end{equation*}
Using some combinatorials techniques and induction. This part of the proof relies complitely on the combinatorial properties of both expressions. Finally, we compute the ECH index of the embedded contact homology of the toric concave lens space $L(n,1)$ and relate it to a combinatorial model. The special properties of the lens spaces allows us to simplify the expression of the capacities and with some more combinatorial tricks we prove in Lemma \ref{lem:capupperbound} that 
\begin{equation*}
    c_k(X_\Omega)\leq \max\{l_\Omega(\Lambda):\mathcal{L}_n(\Lambda)=k\} 
\end{equation*}

Using the three above inequalities both theorems follow.

%\section{Idea of the proof}

%Let $X_\Omega\subset W_p$ be a concave almost toric domain. Then we can extend the usual equation for the $ECH$ capacities in this case in the following manner

%\begin{theo}
 %   The $ECH$ capacities of $X_\Omega$ are given by
  %  $$C_k(X_\Omega)=\max\{l_\Omega(\Lambda):\mathcal{L}(\Lambda)=k\}$$
  %  Here the maximum is taken over concave integral paths $\Lambda$
%\end{theo}

%\begin{proof}
 %   This follows directly from the work of \cite{keon} in the special case of concave generators. In this case the sum of all the concave generators of a certain index are non null homologous. 
%\end{proof}

\section{Combinatorial properties.}
In this section we describe several combinatorial properties that are going to be useful for the rest of the exposition.

\subsection{Fundamental group of the lens spaces.}
\label{sec:fundamentalgroup}
From Definition \ref{defi:lensspace} and the Seifert-Van Kampen Theorem we can see that  $\pi_1(L(n,1))$ can be identified with $\mathbb{Z}^2/\text{span}\{(0,1),(n,1)\}$ which we do for the rest of the paper. Notice also that in this case $\pi_1(L(n,1))$ and $H_1(L(n,1))$ are isomorphic. Notice that the vector $(-1,0)$ generates the group $\mathbb{Z}^2/\text{span}{(0,1),(n,1)}$.

For any $w\in \mathbb{Z}^2$ we can formally associate an homology class $H_1(L(n,1))$. Notice that there exist uniques $k_1$, $k_2$ integers and $0\leq l<n$ such that 
\begin{equation}
    \label{eq:homologyofvector}
    w+l(-1,0)=k_1(n,1)+k_2(0,1) 
\end{equation}
So we can associate $l$ as the homology class of $w$. Similarly, we can associate to any poligonal path $\Lambda$ with vertices on lattice points a homology class $H_1(L(n,1))$. Write $[\Lambda]=\sum_{v\in \text{Edge}(\Lambda)} v$, so if we take $w=[\Lambda]$ in Equation \ref{eq:homologyofvector} we can associate the homology class $l$. 

\subsection{The combinatorial expression of the ECH capacities and the ball packing.}
\label{sec:com<packing}
\begin{lemm}
\label{lem:capellip}
Let $n$ be an integer at least one, we have 
    $$c_k(E_n(a,b))=\max\{l_\Omega(\Lambda):\mathcal{L}(\Lambda)=k\}$$
    where $a,b$ are positive real numbers. 
\end{lemm}

\begin{proof} Write $X_\Omega=E_n(a,b)$ and suppose first $b/a$ is irrational. 

In the $b/a$ is irrational we can describe the capacities also in the following way. There exist exactly one point $(x_0,y_0)$ in $V_n$ such that a line $L$ with bound exactly $k$ with the axis of $V_n$. Then, $c_k(E_n(a,b))=at_1+bt_2$ where $t_1$ is the horizontal distant from the $y$-axis to the point $(x_0,y_0)$ and $t_2$ is the horizontal distant from $(x_0,y_0)$ and the $(n,1)$-axis.

Given $\Lambda$ with $\mathcal{L}(\Lambda)=2k$. Let $L$ be the curve tangent to $\Lambda$ described in. \ref{exa:ellipsoidaction}. We conclude that $l(\Lambda)=c_{k'}(X_\Omega)$ where $k'$ is the number of lattice points contained in the region bounded by $L$ and the axis. By construction $k'\leq k$ therefore $l_\Omega (\Lambda)\leq c_k(X_\Omega)$. We conclude that 
    $$\max\{l_\Omega(\Lambda):\mathcal{L}(\Lambda)=k\}\leq c_k(E_n(a,b))$$
    Now let $L$ be line with that contains the point $(x_0,y_0)$ and enclosed exactly $k$ lattice points with the axis. Let $\Lambda$ be the concave toric domain that is the boundary (minus the axis) of the convex hull of the lattice points not strictly enclosed by the axis. Therefore $l_\Omega(\Lambda)=c_k(X_\Omega)$.

    The case in which $a/b$ is rational follows by continuity as in \cite[Lem. ~2.3]{Choi_2014}
\end{proof}

\begin{lemm} Suposse that $X_\Omega$ is a concave toric domain with weight expansion $w(\Omega)=(w_0;w_1,\dots,w_s)$ then:
    $$\max\{l_\Omega(\Lambda):\mathcal{L}_n(\Lambda)=k\} \leq c_k\left(B_{n}(w_0)\amalg \coprod_{i=1}^s B(w_i)\right)$$     
\end{lemm}

\begin{proof}
We are going to argue inductively over $s$. Notice that our base case $s=0$ is cover by Lemma \ref{lem:capellip} since this are the capacities of the singular ball $B_n(w_0)$.

Lets take $s>0$ are assume that for every $s'<s$ the result holds. 

Writen then $w(\Omega)=(a_0,\dots,a_s)$. Let $\Omega_1,\Omega_2$ and $\Omega_3$ be the regions as defined in the weight expansion. Let $W_1,W_2$ and $W_3$ be the disjoint union of the balls defined by the ball packings of $X_{\Omega_1}, X_{\Omega_2}$ and $X_{\Omega_3}$.

To prove the claim it is enough to prove that for every $\Lambda$ with $\mathcal{L}(\Lambda)=k$ there exist $\Lambda_1,\Lambda_2$ and $\Lambda_3$ with 
\begin{equation}
\label{sumlat}
k_1+k_2+k_3=k    
\end{equation}
where $\mathcal{L}(\Lambda_i)=k_i$ and
\begin{equation}
\label{sumnorm}
  l_{\Omega_1}(\Lambda_1)+l_{\Omega_2}(\Lambda_2)+l_{\Omega_3}(\Lambda_3)=l_{\Omega}(\Lambda)  
\end{equation}

Because it follows by induction hypothesis that 
$$l_{\Omega}(\Lambda)=l_{\Omega_1}(\Lambda_1)+l_{\Omega_2}(\Lambda_2)+l_{\Omega_3}(\Lambda_3)\leq c_{k_1}(W_1)+c_{k_2}(W_2) +c_{k_3}(W_3)$$

therefore
$$l_{\Omega}(\Lambda)\leq c_{k}(W_1\amalg W_2 \amalg W_3) $$

\textit{Construction of $\Lambda_1,\Lambda_2$ and $\Lambda_3$:} 

The construction of $\Lambda_1,\Lambda_2$ and $\Lambda_3$ is similar to the construction of $\Omega_1,\Omega_2$ and $\Omega_3$ obtained in the definition of the weight expansion. We define $\Lambda_1$ to be the longest horizontal line contained in the compact space defined by $\Lambda$ and the lines $\{t(0,1):t\geq 0\}$ and $\{t(p,1):t\geq 0\}$. Denote by $A$ the real number such that $\Lambda_1$ hits $A(0,1)$. Notice that $\Lambda_1$ divide $\Lambda$ in two pieces $\Lambda'_1$ and $\Lambda'_2$. Define $T_2:\mathbb{R}^2\rightarrow \mathbb{R}^2$ as the translation by $(0,-a)$ and $T_3:\mathbb{R}^2\rightarrow \mathbb{R}^2$ as the map obtained when translating by $-a(1,2)$ and then multiplying by $\begin{pmatrix} 0 & 1 \\ -1 & n \end{pmatrix}$. Now we define $\Lambda_2=T_1(\Lambda_2')$ and $\Lambda_3=T_3(\Lambda_3')$.

Writte $k_i=\mathcal{L}(\Lambda_i)$. Notice that the functions $T_2$ and $T_3$ preserve lattice points. It follows that $k_1+k_2+k_3=k$. This proves (\ref{sumlat}).

We have the following relationships between the norms involved:
\begin{align*}
    l_{\Omega_2}(v)=l_{\Omega}(v)-v\times a(0,1) \\
    l_{\Omega_3}(v')=l_{\Omega}(v)-v\times a(n,1)
\end{align*}

Where $v'=\begin{pmatrix} 0 & 1 \\ -1 & n \end{pmatrix} v$.

Summing over the corresponding edges give us
\begin{align*}
    l_{\Omega_2}(\Lambda_2)&=l_{\Omega}(\Lambda'_2)-\displaystyle\sum_{v\in \text{Edges}(\Lambda_2)}v\times a(0,1) \\
    l_{\Omega_3}(\Lambda_3)&=l_{\Omega}(\Lambda'_3)-\displaystyle\sum_{v\in \text{Edges}(\Lambda_3)}v\times a(n,1)
\end{align*}

Notice that the quantity $d'=\sum_{v\in \text{Edges}(\Lambda_2)}v\times a(0,1)$ correspond to the horizontal displacement of $\Lambda$ and $d''=\sum_{v\in \text{Edges}(\Lambda_3)}v\times a(n,1)$ correspond to the horizontal distance between the rightmost point of $\Lambda_3$ and the $(n,1)$-axis.

 By noticing that $\Lambda_1$ consist of only multiplicities of the vector $(-1,0)$, that $l_{\Omega_1}(-1,0)=a=a(0,1)\times (-1,0) =a(n,1)\times (-1,0)$ and Example \ref{exa:ellipsoidaction}, we see that 
 $$l_\Omega(\Lambda_1)=a(d'+d'''+d'')$$ 
 
 Where $d'''$ is the number of times that $e_{-1,0}$ appear in $\Lambda$. By summing up the last three equations we conclude the proof. 
\end{proof}

\subsection{Corounding the corner and the polygonal paths.}

In this section we explain a very important operation over the concave integral paths called corounding the corner. Versions of this operation are common in certain settings, see for example \cite{cristofaro2020proof, hutchings2005periodic, hutchings2006rounding}. 

Let $\Lambda$ be a concave integral path in $V_n$. Take a vertex $p$ of $\Lambda$. Notice that $\Lambda$ with the axis define two regions in $V_n$ one bounded and one unbounded. Lets denote by $A\subset V_n$ the set of lattice points contained in the unbounded region including the lattice points in the concave integral path $\Lambda$. Consider the convex hull $C$ of $A-\{p\}$ and let $\Lambda'$ be the concave integral path that contains the lattice points of the boundary of $C$ that are not in the axis and intersect each axis exactly once. We say that $\Lambda$ is obtained from $\Lambda'$ by \textit{corounding the corner}. See Figure \ref{Coroundingexamples} for some examples.

We now describe an important property of the rounding the corner operation. We begin with the following lemma about the $\Omega$-length:

\begin{lemm}
    Suppose that $X_\Omega$ is a concave toric domain in $M_n$. Let $v_1$ and $v_2$ be pair of vectors. Then 
        \begin{equation}
            \label{ine:inv.triangle.inequality}
            l_\Omega(v_1)+l_\Omega(v_2)\leq l_\Omega(v_1+v_2)    
        \end{equation}
\end{lemm}
    \begin{proof}
    For simplicity lets suppose that $X_\Omega$ is smooth the non- smooth case is similar. Lets suppose that $\partial \Omega$ is parametrized by a function $a:[0,1]\rightarrow V_n$. Consider a  vector $v$ and an auxiliary function $f_v(t)=a\times v$. Lets check that the minimum value of $f_v$ is $l_\Omega(v)$. Notice that $f_v'(t_0)=0$ only if $v$ is parallel to $a'(t_0)$. We have two options $(i)$ there is $t_0$ such that $a'(t_0)$ is parallel to $v$ or $(ii)$ there is not such $t_0$. Lets first check $(i)$, notice that $t_0$ is unique, futhermore, $f(t_0)=l_\Omega(v)$. Since, $f_v''(t_0)=a''(t_0)\times v=-v \times a''(t_0)>0$ the value $l_\Omega(v)$ is the minumum of $f_v$. Now lets check $(ii)$, this would mean that $a'(0)\times v >0 $ or $a'(1)\times v <0$, in the first case the minimun is attained at $t=0$ and in the second case the minimum is attained at $t=1$, in both cases the value of $f_v$ is $l_\Omega(v)$.
    To proof of the inequality \ref{ine:inv.triangle.inequality}, notice that $l_{\Omega}(v_1)+l_\Omega(v_2)\leq f_{v_1+v_2}(t)$ for every $t\in [0,1]$. Therefore, the inequality follows. 
\end{proof}

For the next lemma we need to consider a different family of concave paths. 
\begin{defi}
    A \textit{real path} in $V_n$ is a piecewise linear continuos path $\Lambda$ with starting point in line $\{t(n,1):t\geq 0\}$ and end point in the $y$-axis such that every edge $v$ of $\Lambda$ is proportional to a primitive vector $(p,q)$. We say the real path $\Lambda$ is \textit{concave} if it lies above any of the tangent lines at its smooth points.
\end{defi}

\begin{lemm}
    \label{lem:generalized_triangle_ineq}
    Let $\Lambda$ and $\Lambda'$ be concave real paths. Lets suppose that $\Lambda$ has exactly two edges $w_1$ and $w_2$. Lets suppose that $\Lambda$ and $\Lambda'$ have the same endpoints. Then 
    $$l_\Omega(w_1)+l_\Omega(w_2)\leq l_\Omega(v_1)+\cdots +l_\Omega(v_m)$$
    Where $v_1,\dots v_m$ are the edges of $\Lambda'$.
    
\end{lemm}

\begin{proof}
     Let's asumme that the edges $w_1$ and $w_2$ of $\Lambda$ are  ordered by decreasing slope. We argue by induction over the number of edges $m$ that $\Lambda'$ has. For $m=1$ the result follows directly from inequality \ref{ine:inv.triangle.inequality}. Lets suppose that the result is true for $m$ edges and lets prove it for $m+1$ edges. Lets $v_1,v_2,\cdots v_{m+1}$ be the edges of $\Lambda'$ ordered by decreasing slope. We extend $v_m$ to a vector $v_m'$ that hits $w_2$. The vector $w_2$ is then split in two parts, one part $w_2'$ that intersect $w_1$ and another part $w_2''$ that intersect with $v_{m+1}$. By induction hypothesis we have that 
     $$l_\Omega(w_1)+l_\Omega(w_2')\leq l_\Omega(v_1)+\cdots+ l_\Omega(v_{m-1}) + l_\Omega(v_m')$$
     Lets denote by $v_m''$ the vector $v_m'-v_m$. Then, inequality \ref{ine:inv.triangle.inequality} give us
     $$l_\Omega(v_m'')+l_\Omega(w_2'')\leq l_\Omega(v_{m+1}) $$
     Summing both inequatilities we obtain:
     $$l_\Omega(w_1)+l_\Omega(w_2)\leq l_\Omega(v_1)+\cdots+ l_\Omega(v_{m}) + l_\Omega(v_{m+1})$$
     Since $w_2', w''_2$ and $w_2$ are parallel and $v_m', v''_m$ and $v_n$ are parallel. This finish the inductive step and the proof. 
\end{proof}

\begin{col}
\label{col:cor-red-action}
    Let $X_\Omega$ be a concave toric domain in $M_n$. Suppose that $\Lambda$ is obtained from $\Lambda'$ by corrounding the corner. Then $l_\Omega(\Lambda)\leq l_\Omega(\Lambda')$.
\end{col}

\begin{proof}
    Let suppose that $\Lambda$ is obtained from $\Lambda'$ by corounding the corner. Notice that we can suppose that we have removed every common edge until $\Lambda$ consist of exactly two edges. This leave us with the same conditions of Lemma \ref{lem:generalized_triangle_ineq}. The results follows. 
\end{proof}

\begin{figure}
    \centering
    \input{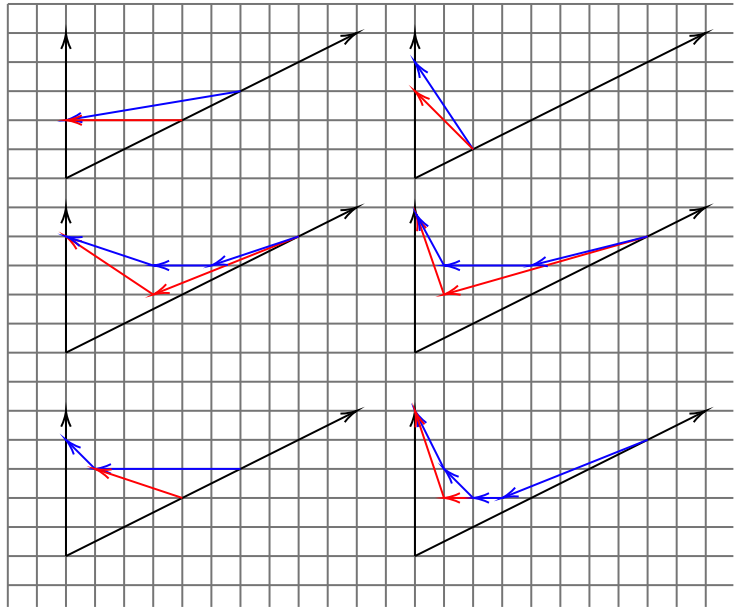}
    \caption{Examples of corounding the corner in $L(2,1)$. The paths $\Lambda'$ and $\Lambda$ are represented by blue and red respectively.}
    \label{Coroundingexamples}
\end{figure}
\section{Foundations of Embedded Contact Homology}
\label{ECHfoundations}

Let $Y$ be a  closed contact $3$-manifold with a contact form $\lambda$, that is,  $\lambda$ is a $1$-form such that $\lambda\wedge\dd \lambda>0$, and let $\xi=\ker \lambda$ be the contact structure. The Reeb vector field $R_{\lambda}$ is the unique vector field in $Y$ satisfying:
\begin{align*}
    \iota_{R_{\lambda}}\dd\lambda=0  \text{ and } & \lambda(R_\lambda)=1 
\end{align*}

We denote by $\phi_t$ the flow of $R_{\lambda}$ which is usually called the Reeb flow. A closed orbit of $\phi_t$ is called a Reeb orbit. A Reeb orbit $\gamma:\mathbb{R}/T\mathbb{Z}\rightarrow Y$ with period $T>0$ is nondegenerate when the linearized return map $P_\gamma:=\dd \phi_T|_\xi: \xi_{\gamma(0)}\rightarrow \xi_{\gamma(0)}$ does not admit $1$ as an eigenvalue. The contact form $\lambda$ is nondegenerate if all Reeb orbits are nondegenerate. Suppose that $\lambda$ is nondegenerate. Since $P_\lambda$ is a linear symplectomorphism, it turns out that the Reeb vector field admits three types of closed orbits:

\begin{enumerate}
    \item \textit{Elliptic}: orbits $\gamma$ such that the eigenalues of the linearized return map $P_\gamma$ are norm one complex numbers.
    \item \textit{Positive hyperbolic}: when the eigenvalues of $P_\gamma$ are positive real numbers.
    \item \textit{Negative hyperbolic}: When eigenvalues of $P_\gamma$ are negative real numbers. 
\end{enumerate}

An orbit set $\alpha=\{(\alpha_i,m_i)\}$ is a finite set, where $\alpha_i$ are distinct embedded Reeb orbits on $Y$ and $m_i$ are positive integers. An admissibe orbit set is an orbit set such that $m_i=1$ whenever $\alpha_i$ is hyperbolic. We denote the homology class of an orbit set $\alpha$ by 
$$[\alpha]=\displaystyle\sum m_i [\alpha_i]\in H_1(Y)$$
For a fixed $\Gamma\in H_1(Y)$, and a generic almost complex structure $J$ on $\mathbb{R}\times Y$ compatible with its symplectic structure, the chain complex $\text{ECC}_*(Y,\lambda,\Gamma,J)$ is the $\mathbb{Z}_2$-vector space generated by the admissible orbits set in homology class $\Gamma$, and its differential counts certain $J$-holomorphic curves in $\mathbb{R}\times Y$, as explained below. This chain complex gives rise to the embedded contact homology $\text{ECH}_*(Y,\lambda,
\Gamma,J)$. Taubes proved \cite{TaubesECHSW}  that $\text{ECH}_*(Y,\lambda,
\Gamma,J)$ is isomorphic to a version of Seiberg-Witten Floer cohomology $\Hat{HM}^{-*}(Y,\mathfrak{s}_\xi+PD(\Gamma))$. In particular, $\text{ECH}_*(Y,\lambda,\Gamma,J)$ does not depend on $\lambda$ or $J$, and so write $ECH_*(Y,\xi,\Gamma)$.

In the next sections we will need to consider relative homological classes that relate two different orbit set $\alpha$ and $\beta$ with the same homology. To be precise we denote by $H_2(Y,\alpha,\beta)$ the affine space over $H_2(Y)$ that consist of $2$-chains $\Sigma$ in $Y$ with 
$$\partial \Sigma=\displaystyle\sum_im_i\alpha_i-\displaystyle\sum_jn_j\beta_j$$
modulo boundaries of $3$-chains. We call $H_2(Y,\alpha,\beta)$ the \textit{relative second homology of $\alpha$ and $\beta$}.

Several of the definitions of ECH are a bit delicate. Because of that we dedicate some more  subsections to properly define the different parts that constitute 
 this homology.

%%%%%%%%%%%%%%%%%%%%%%%%%%%%%%%%%%%%%%%%%%%%%%%%%%%%%%%%%%%%%%%%%%%%%%%%%%%%
\subsection{The \text{ECH} index}

We denote by $H_2(Y,\alpha,\beta)$ the affine space over $H_2(Y)$ that consist of $2$-chains $\Sigma$ in $Y$ with 
$$\partial \Sigma=\displaystyle\sum_im_i\alpha_i-\displaystyle\sum_jn_j\beta_j$$
modulo boundaries of $3$-chains. We call $H_2(Y,\alpha,\beta)$ the \textit{relative second homology of $\alpha$ and $\beta$}.

In this section we define the ECH index which is an interger number associated to a pair of Reeb sets with the same homology and a relative homological class of these two Reeb sets. The ECH index gives the gradding of the embedded contact homology. An interesting feature of the ECH index is that it is the sum of three terms that we define below, the Relative Conley-Zenhder index, the relative Chern Class and the relative intersection number, each of these terms depend on the trivialization of the contatc structure over the Reeb orbits, but the ECH index itself does not.  

Since trivialization play an essential role in the concepts we will introduce, we need to add some notations. We denote by $\mathcal{T}(\gamma)$ the set of homotopy classes of symplectic trivialization of $\xi|_{\gamma}$. This is an affine space over $\mathbb{Z}$: given two trivializations $\tau_1,\tau_2:\xi|_{\gamma}\rightarrow \mathbb{S}^1\times \mathbb{R}^2$, we denote by $\tau_1-\tau_2$ the degree of $\tau_1\circ \tau_2^{-1}:\Sp^1\rightarrow \text{Sp}(2,\mathbb{R})\cong \mathbb{S}^1$. Let $\alpha=\{(\alpha_1,m_i)\},\beta=\{(\beta_j,n_j)\}$ be two orbit sets. If $\tau\in \mathcal{T}(\alpha,\beta):=\prod_i\mathcal{T}(\alpha_i)\times \prod_j \mathcal{T}(\beta_j)$, the elements of $\mathcal{T}(\alpha_i)$ and $\mathcal{T}(\beta_j)$ are denoted by $\tau_i^+$ and $\tau_j^-$.

\subsubsection{Conley-Zenhder index}
\label{Conley-Zenhder}

Now we define the relative Conley-Zenhder index which roughly counts how much the contact structure turn near a Reeb orbit with respect to a certain trivialization. Let $\gamma:\mathbb{R}/T\mathbb{Z}\rightarrow Y$ be a parametrized Reeb and  $\tau$ a trivialization of $\gamma$. If $\phi_t$ is the Reeb flow, the derivative
$$\dd\phi_t:T_{\gamma(0)} Y\rightarrow T_{\gamma(t)}Y$$
restricts to a linear symplectomorphism $\psi_t:\xi_{\gamma(0)}\rightarrow \xi_{\gamma(t)}$. Using the trivialization $\tau$, the later can be viewed as a $2\times 2$ symplectic matrix for each $t$. Since $\lambda$ is nondegenerate, this give rise to a path of symplectic matrices starting at the identity $I_{2\times 2}$ and ending at the linearized return map $\psi_T=P_{\gamma}$, which does not have $1$ as an eigenvalue. So the Conley-Zehnder index $CZ_\tau(\gamma)\in \mathbb{Z}$ is defined as the Conley-Zehnder index of the path $\{\psi_t\}_{t\in [0,T]}$. In dimension four this index can be explicity defined as follows. 
If $\gamma$ is hyperbolic, let $v\in \mathbb{R}^2$ be an eigenvector of $P_\gamma$, then the family of vectors $\{\Psi_t(v)\}_{t\in [0,T]}$  rotates by angle $\pi k$ for some integer $k$ (which is even in the positive hyperbolic case and odd in the negative hyperbolic case), and
\begin{equation}
    \label{eq:CZhyperbolic}
    \text{CZ}_{\tau}(\gamma)=k    
\end{equation}

If $\gamma$ is elliptic, then we can change the trivialization so that each $\psi_t$ is rotation by angle $2\pi \theta_t\in \mathbb{R}$ where $\theta_t$ is a continuous function of $t\in [0,T]$ and $\theta_0=0$. The number $\theta=\theta_T\in \mathbb{R}/\mathbb{Z}$ is called the `rotation angle' of $\gamma$ with respect to $\tau$, and 
\begin{equation}
    \label{eq:CZelliptic}
    \text{CZ}_\tau(\gamma)=2 \lfloor \theta\rfloor+1
\end{equation}
If one changes the trivialization $\tau$ by another $\tau'$, the Conley-Zehnder index changes in the following way:
$$\text{CZ}_\tau(\gamma^k)-\text{CZ}_\tau'(\gamma^k)=2k(\tau-\tau')$$

\subsubsection{Relative first Chern class}

Let $Z\in H_2(Y,\alpha,\beta)$ and $\tau\in \mathcal{T}(\alpha,\beta)$. Given a surface $S$ with boundary and a smooth map $f:S\rightarrow Y$ representing $Z$, the relative first Chern class $c_\tau(Z)=c_1(\xi|_{f(s)},\tau)\in \mathbb{Z}$ is defined as the signed count of zeros of a generic section $\phi$ of $f^*\xi$ that is trivial with respect to $\tau$.

The function $c_\tau$ is linear relative to the homology class, that is If $Z\in H_2(Y,\alpha,\beta)$ and $Z'\in H_2(Y,\alpha',\beta')$ then
\begin{align}
    c_\tau(Z+Z')=c_\tau(Z)+c_\tau(Z')
\end{align}

Moreover, if we change the trivialization $\tau$ by a trivialization $\tau'$, then 
\begin{align}
    c_\tau(Z)-c_{\tau'}(Z')=\displaystyle\sum_im_i({\tau'}_i^+-\tau_i^+)-\displaystyle\sum_j n_j ({\tau'}_j^--\tau_j^-)
\end{align}

\subsubsection{Relative intersection number}
    Let $\pi_Y:\mathbb{R}\times Y\rightarrow Y$ denote the projection and take a smooth map $f:S\rightarrow [-1,1]\times Y$, where $S$ is compact oriented surface with boundary, such that $f|_{\partial S}$ consists of positively oriented covers of $\{1\}\times \alpha_i$ with multiplicity $m_i$ and negatively oriented covers of $\{-1\}\times \beta_j$ with multiplicity $n_j$, $\pi_Y\circ f$ represents $Z$, the restriction $f|_{\Dot{S}}$ to the interior of $S$ is an embedding, and $f$ is transverse to $\{-1,1\}\times Y$. Such an $f$ is called an \textit{admissible representative} for $Z\in H_2(Y,\alpha,\beta)$ and we abuse notation by denoting this representative as $S$. Futhermore, suppose that $\pi_Y|_S$ is an immersion near $\partial S$ and $S$ contains $m_i$ (resp. $n_j$) singly convered circles at $\{1\}\times \alpha_i$ (resp. $\{-1\}\times \beta_j$), given by projecting conormal vectors in $S$, are $\tau$-trivial. Moreover, in each fiber $\xi$ over $\alpha_i$ or $\beta_j$, these sections lie in distinct rays. Then $S$ is a $\tau$-\textit{representative}.

Let $\tau\in \mathcal{T}(\alpha\cup\alpha',\beta\cup\beta')$ be a trivialization and $S, S'$ be $\tau$-representatives of $Z\in H_2(Y,\alpha,\beta)$ and $Z'\in H_2(Y,\alpha',\beta')$ respectively, such that the projected conormal vectors at the boundary all lie in different rays. Then $Q_\tau(Z,Z')\in \mathbb{Z}$  is the signed count of (transverse) intersections of $S$ and $S'$ in $(-1,1)\times Y$. Also $Q_\tau$ is quadratic in the following sense
\begin{equation}
    Q_\tau(Z+Z')=Q_\tau(Z)+2Q_\tau(Z,Z')+Q_\tau(Z')
\end{equation}
If $Z=Z'$ we write $Q_\tau(Z):=Q_\tau(Z,Z)$. It can be proven (see for example \cite{hutchings2014lecture}) that 
\begin{equation}
    \label{relativeQtau}
    Q_\tau(Z)=c_1(N,\tau)-w_\tau(S)
\end{equation}
where $c_1(N,\tau)$, the \textit{relative Chern number of the normal bundle}, is a signed count of zeros of a generic section of $N|_S$ such that the restriction of this section to $\partial S$ agrees with $\tau$; note that the normal bundle $N$ can be canonically identified with $\xi$ along $\partial S$. Meanwhile, the term $w_\tau(S)$, the \textit{asymptotic writhe}, is defined by using the trivialization $\tau$ to identify a neighborhood of each Reeb orbit with $\Sp^1\times D^2\subset \mathbb{R^3}$, and then computing the writhe at $s>>0$ slice of $S$ near the boundary using this identification.

Finally, if $Z,Z'\in H_2(Y,\alpha,\beta)$, changing the trivialization yields
\begin{equation}
    Q_\tau(Z,Z')-Q_\tau(Z,Z')=\displaystyle\sum_i m_i^2 ({\tau'}_i^+-\tau_i^+) -\displaystyle\sum_j n_j^2({\tau'}_j^+-\tau_j^+)
\end{equation}

With the above definitions in place we can properly define the ECH index

\begin{defi}
    Let $\alpha=\{(\alpha_i,m_i)\},\beta=\{(\beta_j,n_j)\}$ be two orbit sets in the homology class $\Gamma$ and $Z\in H_2(Y,\alpha,\beta)$. The $\text{ECH}$ index is defined by 
    \begin{equation}
        \label{ECHindex}
        I(\alpha,\beta,Z)=c_\tau(Z)+Q_\tau(Z)+\text{CZ}_\tau^I(\alpha)-CZ_\tau^I(\beta)
    \end{equation}
    where $\text{CZ}^I_\tau(\alpha)=\sum_i\sum_{k=1}^{m_i}\text{CZ}_\tau(\alpha_i^{k_i})$ and similarly for $\text{CZ}^I_\tau(\beta)$.
\end{defi}

    \begin{pro}\cite[Sec.~3.4]{hutchings2014lecture}
        \label{IndexECHproper}
        The $\text{ECH}$ index has the following properties:
        \begin{enumerate}[a)]
            \item (\textit{Well defined}) $I(\alpha,\beta,Z)$ does not depend on $\tau$, although each term of the formula does.
            \item (\textit{Additivity}) $I(\alpha,\beta,Z+W)=I(\alpha,\delta,Z)+I(\delta,\beta,W)$, whenever $\delta$ is another orbit set in $\Gamma$, $Z\in H_2(Y,\alpha,\beta)$ and $W\in H_2(Y,\delta,\beta)$.
            \item (\textit{Index parity}) If $\alpha$ and $\beta$ are chain complex generators, then 
            $$(-1)^{I(Z)}=\epsilon(\alpha)\epsilon(\beta)$$
            where $\epsilon(\alpha)$ denotes minus to the number of positive hyperbolic orbits in $\alpha$ and similarly $\epsilon(\beta)$.
            \item (\textit{Index ambiguity Formula}) $I(\alpha,\beta,Z)-I(\alpha,\beta,Z')=\langle c_1(\xi)+2 \text{PD}(\Gamma),Z-Z' \rangle$ where $c_1(\xi)$ is the first Chern class of the vector bundle $\xi$ and $\text{PD}$ denotes the Poincare dual.
        \end{enumerate}
    \end{pro}

%%%%%%%%%%%%%%%%%%%%%%%%%%%%%%%%%%%%%%%%%%%%%%%%%%%%%%%%%%%%%%

\subsection{Fredholm index, Differential and Grading}
While the ECH index gives the gradding of the embedded contact homology, the Fredholm index gives the dimension of the moduli space of the J-holomorphic currents that we are intereted in counting. For a generic almost complex structure $J$ the Fredholm index of a $J$-holomorphic curve $C$ is defined as (see \cite[Sec~3.2]{hutchings2014lecture} for details): 
\begin{align}
    \label{FredholmIndex}
    \text{ind}(C)=-\chi(C)+2 c_\tau(C)+\displaystyle\sum_{i=1}^k \text{CZ}_\tau(\gamma_i^+)-\displaystyle\sum_{j=1}^l\text{CZ}_\tau(\gamma_j^-)
\end{align}
where $\chi(C)$ denotes the Euler characteristic of the $J$-holomorphic curve $C$ with $k$ positive ends at the Reeb orbits $\gamma_1^+\dots\gamma_k^+$ and the $l$ negative ends at Reeb orbits $\gamma_1^-\dots\gamma_k^-$.

The proposition below relates the ECH index with the Fredholm index when the ECH index is one or two. This relationship is one of the important step towards the definition of the differential map of the embedded contact and it also allow us to consider the U-map which we define in section \ref{Umap}. Here a trivial cylinder is $\mathbb{R}\times \gamma$, where $\gamma$ is a Reeb orbit.

\begin{pro}{\cite[Prop~3.1]{hutchings2014lecture}}
\label{BasedDiff}
Suppose $J$ is generic. Let $\alpha$ and $\beta$ be orbit sets and let $\mathcal{C}\in \mathcal{M}(\alpha, \beta)$ be any $J$-holomorphic current in $\mathbb{R}\times Y$, not necessarily somewhere injective. Then
\begin{enumerate}
    \item $I(\mathcal{C})\geq 0$, with equality if and only if $\mathcal{C}$ is a union of trivial cylinders with multiplicites. 
    \item If $I(\mathcal{C})=1$ then $\mathcal{C}=\mathcal{C}_0\sqcup C_1$, where $I(\mathcal{C}_0)=0$, and has $\text{ind}(C_1)=I(C_1)=1$.
    \item If $I(\mathcal{C})=2$, and $\alpha$ and $\beta$ are chain complex generators, then $\mathcal{C}=\mathcal{C}_0\sqcup C_2$, where $I(\mathcal{C}_0)=0$, and has $\text{ind}(C_1)=I(C_1)=2$.
\end{enumerate}
\end{pro}

To prove the above Proposition the following property is used which is a particular case of \cite[Prop.~7.1]{hutchings2002index}: If $\mathcal{C}$ is a $J$-holomorphic current with no trivial cylinders and $T$ is the union of (possibly repeated) trivial cylinders, then
\begin{equation}
    \label{intcyl}
    I(C\cup T)\geq I(C)+2\#(C\cap T)
\end{equation}
From intersection positivity it also follows that $\#(C\cap T)\geq 0$, with equality if and only if $C$ and $T$ are disjoint.

Given two chain complex generators $\alpha$ and $\beta$, the chain complex differential $\partial$ coefficient $\langle\partial \alpha, \beta\rangle\in \mathbb{Z}_2$ is a $\text{mod } 2$ count of $\text{ECH}$ index $1$ of $J$-holomorphic curves in the symplectization of $Y$ that \textit{converge as currents} to $\sum_i m_i\alpha_i$ as $s\rightarrow \infty$ and to $\sum_j n_j \beta_j$ as $s\rightarrow -\infty$ see e.g. \cite{hutchings2014lecture}.

It follows from Proposition \ref{BasedDiff} that $I$ gives rise to a relative $\mathbb{Z}_d$-grading on the chain complex $\text{ECC}_*(Y,\lambda,\Gamma,J)$, where $d$ is the divisibility of $c_1(\xi)+2\text{PD}(\Gamma)\in H^2(Y;\mathbb{Z})$ mod torsion. In order to define an (non-canonical) absolute $\mathbb{Z}_d$-grading, it is enought to fix some generatos $\beta$ with homology $\Gamma$ and set
$$I(\alpha,\beta):=[I(\alpha,\beta,Z)],$$
for an arbitraty $Z\in H_2(Y,\alpha, \beta)$. By additivity property 2. in Propostion \ref{BasedDiff} the differential decreases this absolute grading by $1$. Moreover, when $c_1(\xi)+2\text{PD}(\Gamma)$ is torsion in $H^2(Y;\mathbb{Z})$, we obtain a $\mathbb{Z}$ gradding on $\text{ECC}_*(Y,\lambda,\Gamma, J)$ as in the case of lens spaces.

\subsection{Additional structures for embedded contact homology.} 

In this subsection we define some additional important structures in the ECH setting that will be needed in the rest of the exposition. 

\subsubsection{$U$-map}
When $Y$ is connected, there is a well-defined ``U-map"
\begin{align}
    \label{Umap}
    U:\text{ECH}_*(Y,\xi,\gamma)\rightarrow\text{ECH}_{*-2}(Y,\xi,\Gamma)
\end{align}

This is induced by a chain map
\begin{align}
    U_{J,z}:(\text{ECC}_*(Y,\lambda,\Gamma),\partial_J)\rightarrow(\text{ECH}_{*-2}(Y,\xi,\Gamma),\partial_J)
\end{align}    
which counts $J$-holomorphic currents with $\text{ECH}$ index $2$ passing through a generic point $z\in \mathbb{R}\times Y$. The assumption that $Y$ is connected implies that the induced map on homology does not depend on the choice of base point, see \cite[Sec 2.5]{WeinsteinConjecture}  for details. Taubes showed in \cite{TaubesECHSW} Theorem 1.1 that the $U$ map induced in homology  agrees with a corresponding map on Seiberg-Witten Floer homology. We thus obtain the well-defined $U$-map (\ref{Umap}).

\begin{defi}
        A $U$-\textit{sequence} for $\Gamma$ is a sequence $\{\sigma_k\}_{k\geq 1}$where each $\sigma_k$ is a nonzero homogenous class in $\text{ECH}_*(Y,\xi,\Gamma)$, and $U\sigma_{k+1}=\sigma_k$ for each $k\geq 1$.
\end{defi}
We will need the following nontriviality result for the $U$-map, which is proved by combining Taubes' isomorphism with a result from Kromheimer-Mrowka \cite{kronheimer2007monopoles}:

\begin{pro}{\cite[Prop. 2.3]{cristofaro2019torsion}}
    If $c_1(\xi)+2\text{PD}(\Gamma)\in H^2(Y,\mathbb{Z})$ is torsion, then a $U$-sequence for $\Gamma$ exists. 
\end{pro}
%%%%%%%%%%%%%%%%%%%%%%%%%%%%%%%%%%%%%%%%%%%%%%%%%%%%%%%%%%%%%%%%%
\subsubsection{The $\text{ECH}$ partition conditions}

The $\text{ECH}$ partition conditions are a topological type data associated to the pseudoholomorphic curves (and currents) which can be obtained indirectly from certain $\text{ECH}$ index relations. In particular, the covering multiplicities of the Reeb orbits at the ends of the non-trivial components of the pseudoholomorphic curves (and currents) are uniquely determined by the trivial cylinder component information. The genus can be determined by the current's relative homology class.
\begin{defi}{\cite{hutchings2014lecture}} Let $\gamma$ be an embedded Reeb orbit and $m$ a positive integer. We define two partitions of $m$, the \textit{positive partition} $P_\gamma^+(m)$ and the \textit{negative partition} $P_\gamma^-(m)$ as follows
\begin{itemize}
    \item If $\gamma$ is positive hyperbolic, then
    $$P_\gamma^+(m):=P_{\gamma}^-(m):=(1,\dots,1)$$
    \item If $\gamma$ is negative hyperbolic, then
    \begin{align*}
        P_\gamma^+(m):=P_{\gamma}^-(m):=\begin{cases}
            (2,\dots,2) & m \text{ even}\\
            (2,\dots,2,1) & m \text{ odd}
        \end{cases}
    \end{align*}
    \item If $\gamma$ is elliptic then the partitions are defined in terms of the quantity $\theta\in \mathbb{R}/\mathbb{Z}$ for which $\text{CZ}_\tau(\gamma^k)=2\lfloor k\theta \rfloor+1$. We write
    $$P_\gamma^\pm(m):=P_\theta^\pm(m)$$
    with the right hand side defined as follows. 
\end{itemize}
    Let $\Gamma_\theta^+(m)$ denote the highest concave polygonal path in the plane that starts at $(0,0)$, ends at $(m,\lfloor k\theta \rfloor)$, stays below the line $y=\theta x$ and has corners at lattice points. Then the integers $P_\theta^+(m)$ are the horizontal displacements of the segments of the path $\Gamma_\theta^+(m)$ between the lattice points. 
    
    Likewise, let $\Gamma_\theta^+(m)$ denote the lowest convex polygonal path in the plane that starts at $(0,0)$, ends at $(m,\lfloor k\theta \rfloor)$, stays above the line $y=\theta x$ and has corners at lattice points. Then the integers $P_\theta^-(m)$ are the horizontal displacements of the segments of the path $\Gamma_\theta^-(m)$ between the lattice points. 
    
    Both $P_\theta^\pm(m)$ depend only on the class of $\theta$ in $\mathbb{R}/\mathbb{Z}$. Moreover, $P_\theta^+(m)=P_{-\theta}^{-}(m)$.
\end{defi}

%%%%%%%%%%%%%%%%%%%%%%%%%%%%%%%%%%%%%%%%%%%%%%%%%%%%%%%%%%%%%%%%%%%%%%
\subsubsection{Filtered ECH.}
There is a filtration on \text{ECH} which enables us to compute the embedded contact homology via succesive approximations (see theorem 2.17 \cite{nelson2022embedded} ). The \textit{symplectic action} of an orbits set $\alpha=\{(\alpha_i,m_i)\}$ is 
\begin{align*}
    \mathcal{A}(\alpha):=\displaystyle\sum_i m_i\int_{\alpha_i} \lambda
\end{align*}
If $J$ is a $\lambda$-compatible and there is a $J$-holomorphic current from $\alpha$ to $\beta$, then $\mathcal{A}(\alpha)\geq \mathcal{A}(\beta)$ by Stokes' theorem, since $\dd \lambda$ is an area form on such $J$-holomorphic curves. Since $\partial$ counts $J$-holomorphic currents, it decreases symplectic action, that is, 
\begin{align}
\label{Stokes}
    \langle\partial\alpha,\beta\rangle\not=0 \text{ implies } \mathcal{A}(\alpha)\geq \mathcal{A}(\beta)
\end{align}
Let $\text{ECC}_*^L(Y,\lambda,\gamma,J)$ denote the subgroup of $\text(ECC)_*(Y,\lambda,\Gamma,J)$ generated by orbits set of symplectic action less than $L$. Because $\partial$ decreases action, it is a subcomplex. It is shown (See \cite[theo~1.3]{hutchings2013proof}) that the homology of $\text{ECC}_*(Y,\lambda,\Gamma,J)$ is independent of $J$, therefore we denote its homology by 
$\text{ECC}^L_*(Y,\lambda,\Gamma,J)$, which we call filtered $\text{ECH}$.
Given $L<L'$, there is a homomorphism 
\begin{align*}
\iota^{L,L'}:\text{ECH}_*^L(Y,\lambda,\Gamma)\rightarrow\text{ECH}_*^{L'}(Y,\lambda,\Gamma)   
\end{align*}
induced by the inclusion $\text{ECC}_*^L(Y,\lambda,\Gamma)\rightarrow\text{ECC}_*^{L'}$ and independent of $J$. The $\iota^{L,L'}$ fit together into a direct system $(\{\text{ECC}_*^L(Y,\lambda,\Gamma)\}_{L\in \mathbb{R}},\iota^{L,L'})$. Because taking direct limits commutes with taking homology, we have
\begin{align}
\label{directlimit}
\text{ECH}_*(Y,\lambda,\Gamma)=H_*\left(\lim_{L\rightarrow \infty} \text{ECC}_*^L(Y,\lambda,\Gamma,J)\right)=\lim_{L\rightarrow\infty}\text{ECH}^L_*(Y,\lambda,\Gamma)   
\end{align}

%%%%%%%%%%%%%%%%%%%%%%%%%%%%%%%%%%%%%%%%%%%%%%%%%%%%%%%%%%%%%%%%%%%%
\subsubsection{ECH spectrum}
\label{ECHspectrum}
\text{ECH} contains a canonical class defined as follows. Observe that for any nondegenerate contact three-manifold $(Y,\lambda)$, the empty set of Reeb orbits is a generator of the chain complex $\text{ECC}(Y,\lambda,0,J)$. It follows from (\ref{Stokes}) that this chain complex generator is actually a cycle, i.e,
$$\partial\emptyset=0$$

\text{ECH} cobordism maps can be used to show that the homology class of this cycle does not depend on $J$ or $\lambda$, and thus represents a well-defined class 
$$[\emptyset]\in\text{ECH}_*(Y,\xi,0)$$

\begin{defi}
    Let $(Y,\lambda)$ be a closed contact closed $3$-manifold such that $[\emptyset]\not=0\in\text{ECH}(Y,\xi,0)$. We define the \textit{ECH spectrum} as
\begin{align*}
        c_k(Y,\lambda)=\inf\{L:\eta\in \text{ECH}^L_{2k}(Y,\lambda,0), U^k \eta=[\emptyset]\}
    \end{align*}
\end{defi}

\begin{defi}
    A (four-dimensional) \textit{Liouville domain} is a weakly exact symplectic filling $(X,\omega)$ of a contact three-manifold $(Y,\lambda)$.
\end{defi}

\begin{defi} If $(X,\omega)$ is a four-dimensional Liouville domain with boundary $(Y,\lambda)$, define the $\text{ECH}$ $\textit{capacities}$ of $(X,\omega)$ by
$$c_k(X,\omega)=c_k(Y,\lambda)\in [0,\infty]$$
\end{defi}

A justification for this definition can be found in \cite{hutchings2014lecture} in Section 1.5. The definition of capacities can be extended to non-Liouville domains through a limiting argument. 

%It then follows from the Ellipsoid property that 
%\begin{equation}
 %   \label{capballs1}
 %   c_k(B(a))=ad
%\end{equation}
%where $d$ is the unique nonnegative integer such that 
%\begin{equation}
%    \label{capballs2}
%    \displaystyle\frac{d^2+d}{2}\leq k\leq %\displaystyle\frac{d^2+3d}{2}
%\end{equation}

%Another equation that is useful is the explicit equation for the capacities of the disjoint union of balls. Let $a_1,\dots, a_n$ be positive real numbers, then 
%\begin{equation}            
%\label{disjointellipsoid}
%c_k\left(\coprod_{i=1}^n B(a_i)\right)=\max\left\{ \displaystyle\sum_{i=1}^na_id_i:
%\displaystyle\sum_{i=1}^n\displaystyle\frac{d_i^2+d_i}{2}\leq k\right\}
%\end{equation}
%This equations will be useful later. 

%\include{Toric Orbifolds/Toric Orbifolds}
%\include{Capacities of Concave Toric Domains/Capacities of Concave Toric Domains}
%\include{Ball packing/Ball Packing}
\section{Combinatorial Model of the embedded contact complex of concave lens spaces.}
\label{ECHlensspace}

In this section we relate the embedded contact complex of a concave contact form $\lambda$ over a lens space $L(n,1)$ with a combinatorial model.  As it is usual in this context (see \cite[Sec~4.2]{hutchings2014lecture}) Reeb orbits appear in $\Sp^1$-families and the contact form can be perturbed so that this $\Sp^1$-families becomes two orbits. In our situation we do an additional perturbation that we are calling a \textit{concave perturbation} which is going to be useful in simplifying the combinatorial complex.  

\subsection{Reeb Dynamics for toric contact closed $3$-manifolds.}
\label{sec:Reebdynamics}
Let $X_\Omega$ be a concave toric domain in $M_n$. As it is explain in Section \ref{subsec: ToricLensspaces} the boundary $Y_a=\partial \Omega$ has a contact $1$-form $\lambda$ described by a function $a:[0,1]\rightarrow V_n$. Lets suppose additionally $a'(0)$ and $a'(1)$ are not proportional to a primitive vector. Similar to \cite{Choi_2014} Section 3.3 the closed orbits of the Reeb field associated to $\lambda_a$ are given by the following:

\begin{enumerate}[i.]
    \item
    \label{orbit:positive}
    The orbit $e_+$ obtained by the projection $\pi(\{0\}\times \To^2)$ over $L(n,1)$ with action $\mathcal{A}(e_+)=a(0)\times (n,1)$. 
    \item 
    \label{orbit:positive}
    The orbit $e_-$ obtained by the projection $\pi(\{1\}\times \To^2)$ over $L(n,1)$ with action $\mathcal{A}(e_-)=a_2(1)$.
    \item 
    \label{orbit:torus}
    For each $x\in (0,1)$ for which $(a_1'(x),a'_2(x))$ is proportional to $(-p,q)$ where $p$ and $q$ are relative primes to each other, there is a Morse-Bott $\mathbb{S}^1$-family of Reeb orbits foliating $\{x\}\times \To^2$, with relative homology over $\To^2$ equal $(q,p)$. Each orbit of this folliation has action $a(x)\times (-p,q)$. We denote the foliated torus as $\mathcal{T}_{q,p}$.
\end{enumerate}

The orbits $e_+$ and $e_-$ are called \textit{exceptional orbits}. These orbits are homologoues in $H_1(Y_1)$ and are generators of the singular homology as can be seen from \ref{sec:fundamentalgroup}. Notice that the actions of the Reeb orbits just decribed coincide with the $\Omega$-length given in the Definition \ref{def:omegalength}. 

%%%%%% There is slight problem here that is sadly spread all over this article and it is that several times there is an abuse of language. The rotation of 90 degrees when considering the the orbital map has little to none consequences in the final result, but for sure I would like to feel more confident while solving this problem. For example, in the construction of the surface over which we can calculate the Qtau there is a slight problem with the homology. Notice that via the calculation the homology of the surface there is a point in which by construction the homology of the collpased disk is (n,m) but the quotiente collapses the orbits of the form (-m,n). So... what happened here?

%\begin{rem}
%\label{notationorbits}
% The notations $e_{p,q}, h_{p,q}$ means the elliptic or hyperbolic Reeb orbit with homology $(p,q)$ respectively. We can also write $e_x, h_x$ to mean the elliptic or hyperbolic Reeb orbit ocurring at $x$ where $a'(x)$ is proportional to a primitive vector. We write $e^m_x$ to mean the ellitic orbit at $x\in [0,1]$ with multiplicity $m$ and $h^m_x$ to mean the orbit set $\{(h_x,1),(e_x,m-1)\}$. So a Reeb orbit set $\alpha=\{(\alpha_i,m_i)\}$ can be written with multiplicative notation in a unique way as $\alpha=f_{x_1}^{m_1}\cdots f_{x_k}^{m_k}$ with $x_1< \cdots< x_k$ where each $f_i$ is a label `$e$' or `$h$'.
%\end{rem}

\subsection{Two Steps Perturbation.}
\label{perturbation}

Suppose that $Y_a=L(n,1)$ is a concave toric lens space with a contact form $\lambda_a$. To obtain a simple version of a combinatorial complex it is convinient to do two perturbations over $\lambda_a$. In Proposition \ref{pro:equicomgeo} we stablish a bijection between Reeb orbits sets and integral concave paths

Before explaining the perturbations it is best to calculate the rotation numbers of the exceptional orbits $e_+$ and $e_-$.

\textbf{Trivialization over the elliptic orbits $e_+$ and $e_-$:}

Recall that $M_1$ is diffeomorphic to $\mathbb{C}^2$. That means that a toric domain $X_\Omega$ in $M_1$ is an usual toric domain an $\partial X_\Omega$ is diffeomorphic to $\mathbb{S}^3$. Suppose that the $\partial X_\Omega$ is describe by a function $a$ where $a'(0)$ and $a'(1)$ are not proportional to a primitive vector. Therefore by \cite[Sec.~3.3]{Choi_2014} the Reeb orbits $e_+$ and $e_-$ are non-degenerated and from Equation 3.11 and Equation 3.12 of \cite[Sec.~3.3]{Choi_2014}, we have 

\begin{align}
\label{eq:rotanglesphere}
    \phi_+=-\displaystyle\frac{a'_1(0)}{a'_2(0)} &&
    \phi_-= -\displaystyle\frac{a'_2(0)}{a'_1(0)}
\end{align}

Where, $\phi_+$ and $\phi_-$ are the rotation angles of the orbits $e_+$ and $e_-$ in $Y_a=\partial X_\Omega$ respectively. 

Here the trivialization $\tau$ is given by Equations \ref{eq:Reebandstructure} for the case of the sphere. The behavious of $e_+$ and $e_-$ is similar for toric lens space different from the sphere as it is explain in the Lemma below. 

\begin{lemm}
\label{lem:rotnumbers}
Let $X_\Omega$ be a concave toric domain in $M_n$. Let $Y_a=\partial X_\Omega$ for some function $a$ and suppose that $a'(0)$ and $a'(1)$ are not proportional to a primitive vector. 
The orbits $e_+$ and $e_-$ are elliptic orbits. Futhermore, Under the trivialization induced by Equation \ref{eq:Reebandstructure} we have that 
\begin{align}
    \phi_+=-\displaystyle\frac{a'_2(0) }{a'(0)\times (n,1)}, && \phi_-=-\displaystyle\frac{a_2'(1)}{a_1'(1)}
\end{align}
Where $\phi_+$ and $\phi_-$ are the rotation numbers of $e_+$ and $e_-$ respectively. 
\end{lemm}
\begin{proof}
    %For the orbit $e_2$ the result follows by noticing that a neighborhood of this orbits is strictly contactomorphic to a neighborhood of the sphere $\Sp^3$ with the appropiate contact structure. 
    Lets begin checking that $e_-$ is an elliptic orbit and then we calculate the rotation angle.
    Consider a function $\tilde{a}:[0,1]\rightarrow \mathbb{R}^2$ that is equal to $a$ in a vicinity of $1$ such that $Y_{\tilde{a}}$ is a toric contact manifold. Notice that we $Y_a$ can use the identity matrix to construct a local contactomorphism from a toric open set $U$ of $e_-$ in $Y_a$ and  a toric open set $\tilde{U}$ of $e_-$ in $Y_{\tilde{a}}$. Since, $a'(1)$ is not proportional to primitive vector $e_-$ is a non-degenerated elliptic orbit. This contactomorphism take trivizalization of 
    $\xi_a$ over $e_-$ in $Y_a$ to trivizalization 
    of $\xi_{\tilde{a}}$ over $e_-$ in $Y_{\tilde{a}}$. Therefore, $\phi_+=-\frac{a_2'(1)}{a_1'(1)}$.
    
    We use a similar strategy to check the result for $e_+$. Consider the function $\tilde{a}:[0,1]\rightarrow \mathbb{R}^2$ define as $
    \tilde{a}(t)=(a_2(t), a(t)\times (n,1))$ when $t$ is near $0$ and such that $Y_{\tilde{a}}$ is a contact toric manifold anywhere else. Consider, the matrix 
    $$
    A_n=\begin{pmatrix}0 & 1 \\-1   & n  \end{pmatrix}    
    $$
    Notice that $A_n$ is an  $\text{SL}_2(\mathbb{Z})$ matrix, therefore, the map $(A_n,A_n^{-T})$ can be used to induce a local contactmorphism from a vinicity of $e_+$ in $Y_a$ and to a vinicity of $e_+$ in $Y_{\tilde{a}}$ that takes the trivialization of $e_+$ in $Y_a$ to the trivialization of $e_+$ in $Y_{\tilde{a}}$, therefore, the rotations numbers are the same. The result then follows from Equation \ref{eq:rotanglesphere}.
\end{proof}

From now on we assume that the trivilizations over the contact structures over $e_+$ and $e_-$ are given by Equation \ref{eq:Reebandstructure}.

\subsubsection{Perturbation over the concavity.}
Now we explain how to do a perturbation over the concavity.
    
    \begin{lemm}
    \label{lem:convpertur}
        Let $X_\Omega$ be a concave toric domain in $M_n$. Set $Y_a=(\partial X_\Omega,\lambda_a)$ for some function $a:[0,1]\rightarrow V_n$. Then, for every,  $N$ non-negative integer there exist a perturbation $\tilde{a}$ of $a$ equal to $a$ outside of an open set contaning the exceptional orbits, such that, $e_+$ and $e_-$ are non-degenerated for $Y_{\tilde{a}}$ and such that $CZ_\tau(e^k_+)>N$ and $CZ_\tau(e^k_-)>N$, for every, $k$ positive. 
    \end{lemm}

    \begin{proof}   
        For $\epsilon>0$ lets consider a family of function $\{a^\epsilon\}_{\epsilon>0}$ such that 
 \begin{itemize}
    \item $a_\epsilon'(0)$ and $a'_\epsilon(1)$ are not proportional to a primitive vector. 
    \item $a^\epsilon(x)\times (a^\epsilon)'(x)<0$ for every $x\in [0,1]$
     \item $a(x)=a^\epsilon(x)$ for every $x\in [\epsilon,1-\epsilon]$.
     \item There exist positive constants $k_1$ and $k_2$ such that $a'_\epsilon(0)\rightarrow -k_1 (n,m)$ and $a'_\epsilon(1)\rightarrow k_2 (0,1)$ when $\epsilon\rightarrow 0$.
 \end{itemize}

    By Equation \ref{eq:CZelliptic} and Lemma \ref{lem:rotnumbers} there exist  $\epsilon>0$ small enough such the conditions of the lemma follows. 
    \end{proof}

    \begin{defi}
        For a concave domain $X_\Omega$ we call the perturbation of Lemma \ref{lem:convpertur} a \textit{perturbation over the concavity} or a concave perturbation. 
    \end{defi}

\subsubsection{Morse-Bott Perturbation.}
\label{Morse-Bott}

Let $\{x\}\times \To^2= \mathcal{T}_{q,p}$ be a Morse-Bott torus as explained in \ref{sec:Reebdynamics}. As it is customary in this context (see \cite[Section~4.2]{hutchings2014lecture} or \cite{Choi_2014}) given $L>0$ such that $a(x)\times (-p,q)<L$ the contact form $\lambda_a$ can be perturbed on an arbitrary small vicinity of $\mathcal{T}_{q,p}$ so that the foliated torus $\mathcal{T}_{q,p}$ splits into two embedded orbits of nearly the same action and in no other orbit of action less than $L$ is created by this perturbation. Futhermore the newly created orbits are reparametrizations of two Reeb orbits in the foliation of $\mathcal{T}_{q,p}$. The orbits behave as follow, one is elliptic $\gamma_e$, and the other one is  hyperbolic $\gamma_h$. The perturbation can be chosen so that the actions of $\gamma_e$ and $\gamma_h$ are $(1/L)$-close to $a(x)\times(p,q)$. Furthermore, the perturbation may be arranged so that the linearization is conjugate to a small negative rotation. In particular, from Equation \eqref{eq:CZhyperbolic} and Equation \eqref{eq:CZelliptic} it follows that
\begin{align}
    \text{CZ}_{\tau}(\gamma_e^k)=-1 &\text{ for } \mathcal{A}(\gamma_e^k)<L, \text{ and},\\  
    \text{CZ}_{\tau}(\gamma_h^k)=0 &\text{ for } \mathcal{A}(\gamma_h^k)<L
\end{align}

%Combining this Morse-Bott perturbation with the perturbation of the concavity we have the following useful result. 

%\begin{lemm}
 %   \label{lem:theperturbation}
 %   Let $X_\Omega$ be concave toric domain in $M_n$. Write $Y_a=(\partial X_\Omega,\lambda_a)$. For each $L$ positive real number and $N$ positive integer there exist a perturbation $\lambda^{L,N}$ of $\lambda_a$ such that 
%    \begin{enumerate}[i]
        %\item The exceptional Reeb orbits are non-%degenerated and the relative Conley-Zhender number of all its iterates is bigger than $N$.
        %\item Every orbit that is not an exceptional orbit of action strictly less than $L$ is non-degenerated. Futhermore, this set consist of orbits obtained by  the Morse-Bott perturbation of torus foliated by Reeb orbits as explain at the beggining of this section.
    %\end{enumerate}
%\end{lemm}

%\begin{proof}
 %   First over $\lambda_a$ we do a concave perturbations for the given $N$. Then, we do a Morse-Bott perturbation for every torus foliated by Reeb orbits whose action is less than $L$.
%\end{proof}

\subsubsection{A filtrated ECC associated to a concave toric lens space.}
\label{sec:filtrateeccforconcavetoricdomain}

Let $X_\Omega$ be a concave toric domain. Write $Y_a=(\partial X_\Omega,\lambda)$. In this section we explain how to associate a filtrated embedded contact complex to $Y_a$ and therefore to $X_\Omega$. By the discussion in Section \ref{sec:Reebdynamics} by the exception of the orbits $e_+$ and $e_-$ each Reeb orbit $\gamma$ of $Y_a$ belongs to a torus $\mathcal{T}_{q,p}$ foliated by Reeb orbits, where $(q,p)$ is the singular homology of $\gamma$. It makes sense to formally define the following sets:

\begin{defi}
    A torus set $\mathcal{T}=\{(\mathcal{T}_i,m_i)\}$ is a set of pairs $(\mathcal{T}_i,m_i)$ such that $\mathcal{T}_i$ is or (i) a torus of the form \ref{orbit:torus} in Section \ref{sec:Reebdynamics} or (ii) is one of the exceptional orbits $e_+$ and $e_-$.
\end{defi}

Since every action of every Reeb orbit in a torus $\mathcal{T}_(q,p)$, we can define the action of torus $\mathcal{A}(\mathcal{T}_{q,p})$ as the action of any of its periodic orbits. In the case of the exceptional orbits we keep the same notion of the action. To each torus set $\mathcal{T}=\{(\mathcal{T}_i,m_i)\}$ we can associated an action 
$$\mathcal{A}(\mathcal{T})=\displaystyle\sum_i m_i\mathcal{A}(\mathcal{T}_{q,p})
$$

\begin{defi}
    Let $L>0$ be a positive real number and $N$ a positive integer. Suppose that $\mathcal{A}(\mathcal{T})<L$. A Morse-Bott perturbation of the contact structure $\lambda$ over a torus set $\mathcal{T}$ with respect to $N$ and $L$ is a Morse-Bott pertubation over all of the torus sets contained in $\mathcal{T}$ and a concave perturbation with respect to $N$ in the case that $\mathcal{T}$ contained exceptional orbits. 
\end{defi}

\begin{lemm}
    \label{lem:therealperturbation}
    Let $L$ be a positive real number and $N$ a positive integer. Then there exist a perturbation $\lambda^{L,N}$ arbitrary close to $\lambda$ such that:  
    \begin{enumerate}[i.]
        \item If $\alpha$ is a generator of $ ECC^L(Y,\lambda,0)$ then $\alpha$ consist precisely of orbits sets obtained from a Morse-Bott perturbation on each torus set $\mathcal{T}$ such that $\mathcal{A}(\mathcal{T})<L$.
        \item The Conley-Zhender index of the exceptional orbits and all its iterates is bigger than $N$.
    \end{enumerate}
\end{lemm}
\begin{proof}
    This consist of doing a Morse-Bott perturbation with respect to $L$ and $N$ for each torus set $\mathcal{T}$ such that $\mathcal{A}(\mathcal{T})<L$.
\end{proof}

We denote the perturbation of the manifold by $Y_{L,N}$.

\begin{rem}
    In principle the more correct notation of the embedded contact complex we are studying is $ ECC^L(Y,\lambda,0,J)$ for some generic almost complex structure, since the differential depends on the almost complex structure to be defined. However, since we are not occupaying ourselves with arguments that involve the differential in much detailed, the simplified notation $ ECC^L(Y,\lambda,0,J)$ is justified. 
\end{rem}

\subsection{Concave Generators and the Combinatorial ECH index}
\label{sec:generators}
%\textcolor{red}{To feel that this have been done properly I need to move the conventions so the decorations really match. }

Now we modify slightly the definition of the concave integral path to obtained the notion of concave generator. We also relate it to certain orbits sets of Perturbation given by Lemma \ref{lem:therealperturbation}.

%So to each $\alpha$ orbit Reeb set we can associate a unique concave generator.

\begin{defi}
    A concave generators $\Lambda$ in $V_n$ is a concave integral path in $V_n$ such that each edge $v$ have a label `e' or `h'.
\end{defi}

Let $\Lambda$ be a concave generator in $V_n$. We define the combinatorial ECH index of $\Lambda$ as
 
\begin{equation} 
\label{eq:echinccon}
I(\Lambda)=2\mathcal{L}_{n}(\Lambda)+h(\Lambda)  
\end{equation}

Where $h(\Lambda)$ and $\mathcal{L}(\Lambda)$ is the number of lattice points contain in the closed space defined by $\Lambda$ and the axis of $V_n$ without counting the lattice points of $\Lambda$.

Now lets consider a concave toric domain $X_\Omega$ in $M_n$. For $Y=(\partial X_\Omega,\lambda)$ and let $\lambda^{L,N}$ be a perturbation as given by Lemma \ref{lem:therealperturbation} for some $L$ and some $N$. Suppose that $\alpha=\{(\alpha_i,m_i)\}$ is a orbit set in $ECC^L(Y,\lambda^{L,N},0)$ which does not contain the orbits $e_+$ or $e_-$. For each $\alpha_i$ in the orbit set $\alpha$ write $[\alpha_i]=(q_i,p_i)\in H_2(\To^2)$. Notice that we can organize the orbit set $\alpha$ as $\{(\alpha_1,m_1),\dots,(\alpha_k,m_k)\}$ where $q_1/p_1<\cdots <q_k/p_k$. Notice that the homological condition $[\alpha]=0\in H_2(\To^2)$ implies that there exists a unique path concave path $\Lambda$ such that the edges are consecutive concatenations of the vectors $m_1(-p_1,q_1)\dots m_k(-p_k,q_k)$ and $\Lambda$ begins at the $(n,1)$-axis and ends at the $y$-axis.

\begin{rem}
Notice that each orbit $\alpha_i$ with homology $(q,p)$ in the Reeb orbit set $\alpha$ is rotated 90 degrees clockwise with respect to the corresponding edge $(-p,q)$ in the concave generator $\Lambda$.
\end{rem}

\subsection{A Perturbation Proposition.}
\label{corounding}

The following Proposition explain the relationship between the Reeb orbits sets of a concave toric lens space and the concave polygonal paths in $V_n$.

\begin{pro}
\label{pro:equicomgeo} Let $X_\Omega$ be a concave toric domain in $V_n$ and let $\lambda$ be the contact form over $\partial X_\Omega$. For each $k_0$ non-negative integer there exist $L_0>0$ such that if $L>L_0$ there exist a perturbation $\lambda^{L,k_0}$ arbitrary closed to $\lambda$ such that:
\begin{enumerate}[(a)]
    \item Every orbit with period less than $L$ is non-degenerate.
    \item For every $k\leq k_0$ we have that no orbit set $\alpha \in \text{ECC}_{k_0}^{L} (Y,\lambda^L)$ contains $e_+$ or $e_-$ with any multiplicity.    
    \item For every $k\leq k_0$, the map $\alpha\mapsto \Lambda$ is a bijection between the generators of $\text{ECC}_{k}^{L} (Y,\lambda_\epsilon,0)$ and the set $\{\Lambda:I(\Lambda)=k\}$. Futhermore $I(\alpha)=I(\Lambda)$ and $|\mathcal{A}(\alpha)-l_\Omega(\Lambda)|<1/L$. In particular, the set $\{\alpha:I(\alpha)=k\}$ is finite. 
\end{enumerate}
\end{pro}

In case no explicit mention of the integer $k$ is needed we simple write $\lambda^{L}$.

\begin{defi}
    Given a concave toric domain $X_\Omega$ in $M_n$ and $L>0$. We say that a perturbation given by Proposition \ref{pro:equicomgeo} is an $L$-flat perturbation of $X_\Omega$. In case an explicit mention of the index $k$ we are working on we write $X^L$ is a $L$-flat perturbation with respect to $k$.
\end{defi}

\begin{rem}
    Proposition \ref{pro:equicomgeo} can be improved to describe the ECH differential of concave toric lens space. However, we do not take that approach here.
\end{rem}

The following three subsections are dedicated to the proof of Proposition \ref{pro:equicomgeo}. An outline of the proof is as follows, we choose a perturbation of $X_\Omega$ based on the perturbations explained in Section \ref{sec:filtrateeccforconcavetoricdomain}. The crucial part is to check that the ECH index for this perturbation coincides with the combinatorial index defined in \ref{sec:generators}. In Section \ref{sec:compindex} we calculate the ECH index and argue that every part of the Proposition follows. For complete details of the proof see the end of Section \ref{sec:compindex}.

\subsection{Computations of the index}
\label{sec:compindex}
For $L>0$ and some $N$ positive integer let $\lambda^{L,N}$ a perturbation as the one given by Lemma \ref{lem:therealperturbation}. We begin by computing the $\text{ECH}$ index of an orbit set $\alpha=\{(\alpha_i,m_i)\}$ contained in $\text{ECC}^L(Y_{L,N},\lambda^{L,N},0)$.

\subsubsection{$\text{ECH}$ index}
\label{ECHindexconcave}

Before, we do the calculations we need to construct an auxiliary path $\Lambda$. Let $\alpha=\{(e_+,m_+)\}\cup\{(\alpha_i,m_i)\}\cup\{(e_-,m_-)\}$ where $\{(\alpha_i,m_i)\}$ does not contain exceptional orbits. Let $\alpha'=\{(\alpha_i,m_i)\}$,  we associate to $\alpha'$ a concave generator $\Lambda'$ as explained in subsection \ref{sec:generators}. As explain in Section \ref{sec:fundamentalgroup} there exist unique $k$ and $l$ integers such that

\begin{equation}
    \label{eq:homologyauxiliary}
    m_1(-1,0)+[\Lambda']+m_2(-1,0)=k(0,1)+l(n,1)
\end{equation}

We define the auxiliary path $\Lambda$ by adding $m_1$ times the vector $(-1,0)$ at the right of $\Lambda'$ and adding $m_2$ times the vector $(-1,0)$ at the left of $\Lambda'$. This auxiliary path has zero homology in $H_1(Y)$ as we can see from Equation \ref{eq:homologyauxiliary}. Notice that this path it is not concave nor convex. We think of $\Lambda$ as beggining at the $(n,1)$-axis and ending in the $y$-axis, notice that it is exactly one way this can be done such that $\Lambda$ is completely contained in $V_n$. 
%\begin{figure}
 %   \centering
  %  \input{Drawings/FakePolygonalPath}
  %  \caption{An example of a path $\bar{P}_\alpha$ where $\alpha=\{e_+^3,e_{(-1,0)},e_{(-2,3)},e_-^2\}$. In red arrows we represented the part of the path corresponding to $e_+$ and $e_-$. The rest of the path is in blue.}
 %   \label{PPalpha}
%\end{figure}

\begin{lemm}
    \label{explicitindex}
    Suppose that $X_\Omega$ is a concave toric domain in $M_n$. Let $L>0$ be a real number and $N$ a positive integer. Consider the perturbation $Y_{L,N}$ of $Y=\partial X_\Omega$ given by Lemma \ref{lem:therealperturbation}. 
     Let $\alpha=\{(e_+,m_+)\}\cup\{(\alpha_i,m_i)\}\cup\{(e_-,m_-)\}$ be an  orbit set in $\text{ECC}^L(Y_{L,N},\lambda_{L,N},0)$ and $\Lambda$ be the auxiliary path explained at the begining of this section. Then $I(\alpha,Z)$ does not depend on $Z\in H_2(\alpha,\emptyset,Z)$, futhermore
     \begin{enumerate}       
        
         \item (\textit{Relative Chern Class}) $c_\tau(\alpha)=c_1+c_2$ where $c_1\in \mathbb{Z}$ is the maximal integer such that $c_1(n,1)$ is contained in $\Lambda$. Analogously $c_2\in \mathbb{Z}$ is the maximal integer such that $c_2(0,1)$ is contained in $\Lambda$.

         \item (\textit{Relative self-intersection}) $Q_{\tau}(\alpha)=2A(\Lambda)$ where $A(\Lambda)$ denotes the area of the region defined by $\Lambda$ and the axis.

        \item (\textit{Conley-Zehnder Number}) Denote by $e$ the total number of elliptic orbits in $\{(\alpha_i,m_i)\}$. Then

         \begin{equation}
         \label{CZeq}
         CZ_{\tau}(\alpha)=-e+m_++m_-+ 2\displaystyle\sum_{i=1}^{m_+}\left\lfloor -i\displaystyle\frac{a_2'(0)}{a'(0)\times (n,1)} \right\rfloor+2\displaystyle\sum_{i=1}^{m_-}\left\lfloor -j \displaystyle\frac{a_2'(1)}{a_1(1)} \right\rfloor
         \end{equation}
         
    \end{enumerate}
        
\end{lemm}

\begin{proof}
Since $H_2(Y)=0$ the ECH index does not depend on the the relative homology $Z\in H_2(\alpha,\emptyset,Y)$. Now we construct a surface $S$ such that $\partial S=\alpha$.

The construction of the $S$ is a very classical argument and can be found in different forms in \cite{Choi_2014,cristofaro2020proof,hutchings2014lecture} and others. Here we modify that construction to fit our case. 
  We can construct a surface $S$ in $[-1,1]\times Y$ such that $[S]\in H_2(\alpha,\beta,Y)$. Then we use this manifold to compute $Q_\tau$ and $C_\tau$.
        
    \textit{Construction of the surface $S$:} 
        Consider the projections $\pi:[0,1]\times \To^2\rightarrow L(n,1)$ and consider the natural lifts of the orbits $\alpha'_i$. For the orbit $e_+$ we choose any orbit $e'_+$ in $\{0\}\times \To^2$ with homology $(0,1)$ and $e_-$ choose any orbit $e_-'$ in $\{1\}\times \To^2$ with homology $(0,1)$. Denote by $0=x_+<x_1<\cdots <x_M<x_-=1$  with $M=\sum m_i$ and each $x_i$ represent the point ${x_i}\times \To^2$ at which $\alpha_i$ appears.

        We construct this surface in three steps.
        
        \textbf{Step 1:} \textit{Disjoint Cylinders.} We now describe a construction of disjoint cylinders $\mathcal{C}$. At level $\{1\}\times [0,1]\times \To^2$ we realize the following procedure: for each $\alpha_i$ with multiplicity $m_i$  to obtained a family of trivial cylinders in $\mathbb{R}\times [0,1]\times \To^2$. Choose $m_i$ points $x_{i1},\dots,x_{im_i}$ in a small neighborhood of $x_i$ and not containing any other $x_j$ with $i\not=j$. For each $x_{ik}$ choose an orbit with homology $[\alpha_i$] disjoint from all the others. We do the same for the orbits $e'_+$ and $e'_-$. 
        By following the $s$ direction downwards up to $\{0\}\times \{0\}\times \To^2$ we obtained a set $\mathcal{C}_1$ of disjoint cylinders. 
                
        \textbf{Step 2:} 
        \textit{Construction of the surface $S'$.}
        By the homological conditions we have that 
        $$[\alpha]=-c_1\,(1,-n)+c_2\,(0,1)$$
        Begin with $c_1$ disjoint orbits with homology $(n,1)$ in $\{0\}\times \{0\}\times \To^2$ move this orbits in the $x$ direction forming horizontal cilinders. Each time these cylinders encounter a vertical cilinder we realize negative surgeries similar to \cite{hutchings2005periodic}, in that way we resolve the singularities. After crossing every vertical cylinder we have $c_2$ cylinders in the $x$ direction with homology $(0,1)$. To end this step we make a slighty perturbation near $\{1\}\times[0,1]\times \To^2$ such that $\alpha$ is the boundary on $\{1\}\times[0,1]\times \To^2$ of the obtained surface $S'$. 

        \textbf{Step 3:} \textit{Projecting the surface $S'$ to obtain $S$.} Consider the projection of $S'$ by the quotient map of $\pi:\mathbb{R}\times [0,1]\times \To^2\rightarrow \mathbb{R}\times L(n,1)$. Note that the $c_1$ cilinders in the $x$ direction with homology $(1,-n)$ colapses into disks. Similarly, the $c_2$ cylinders in the $x$ direction with homology $(0,1)$ also colapses into disks. 

        This ends the construction of the surface $S$. 
        
        We now use this surface to compute $c_\tau$ and $Q_\tau$. See figure \ref{TheSurface} for a schematic picture of the surface $S'$ when projected into $[-1,1]\times [0,1]$.

        \begin{figure}
            \centering
            \tikzset{every picture/.style={line width=0.75pt}} %set default line width to 0.75pt        

\begin{tikzpicture}[x=0.55pt,y=0.55pt,yscale=-1,xscale=1]
%uncomment if require: \path (0,202); %set diagram left start at 0, and has height of 202

%Straight Lines [id:da7682180213515117] 
\draw    (39,112) -- (226,111.01) ;
\draw [shift={(228,111)}, rotate = 179.7] [color={rgb, 255:red, 0; green, 0; blue, 0 }  ][line width=0.75]    (10.93,-3.29) .. controls (6.95,-1.4) and (3.31,-0.3) .. (0,0) .. controls (3.31,0.3) and (6.95,1.4) .. (10.93,3.29)   ;
%Straight Lines [id:da7396572133334065] 
\draw    (62,179) -- (61.01,19) ;
\draw [shift={(61,17)}, rotate = 89.65] [color={rgb, 255:red, 0; green, 0; blue, 0 }  ][line width=0.75]    (10.93,-3.29) .. controls (6.95,-1.4) and (3.31,-0.3) .. (0,0) .. controls (3.31,0.3) and (6.95,1.4) .. (10.93,3.29)   ;
%Straight Lines [id:da35807895246729404] 
\draw [color={rgb, 255:red, 4; green, 4; blue, 253 }  ,draw opacity=1 ]   (79,111) -- (79,21) ;
%Straight Lines [id:da554454754049917] 
\draw [color={rgb, 255:red, 4; green, 4; blue, 253 }  ,draw opacity=1 ]   (71,111) -- (71,21) ;
%Straight Lines [id:da005328788255359207] 
\draw [color={rgb, 255:red, 4; green, 4; blue, 253 }  ,draw opacity=1 ]   (112,112) -- (112,21) ;
%Straight Lines [id:da12899845889601647] 
\draw [color={rgb, 255:red, 4; green, 4; blue, 253 }  ,draw opacity=1 ]   (120,113) -- (121,22) ;
%Straight Lines [id:da6473533467352752] 
\draw [color={rgb, 255:red, 4; green, 4; blue, 253 }  ,draw opacity=1 ]   (127,112) -- (127,22) ;
%Straight Lines [id:da9241131514862042] 
\draw [color={rgb, 255:red, 4; green, 4; blue, 253 }  ,draw opacity=1 ]   (218,111) -- (218,21) ;
%Straight Lines [id:da3436667891117464] 
\draw [color={rgb, 255:red, 4; green, 4; blue, 253 }  ,draw opacity=1 ]   (208,111) -- (208,21) ;
%Straight Lines [id:da9135951153753241] 
\draw [color={rgb, 255:red, 4; green, 4; blue, 253 }  ,draw opacity=1 ]   (200,111) -- (200,21) ;
%Straight Lines [id:da4113398662088681] 
\draw    (338,114) -- (524,114) ;
\draw [shift={(526,114)}, rotate = 180] [color={rgb, 255:red, 0; green, 0; blue, 0 }  ][line width=0.75]    (10.93,-3.29) .. controls (6.95,-1.4) and (3.31,-0.3) .. (0,0) .. controls (3.31,0.3) and (6.95,1.4) .. (10.93,3.29)   ;
%Straight Lines [id:da4468668719188049] 
\draw    (351,180) -- (350.01,20) ;
\draw [shift={(350,18)}, rotate = 89.65] [color={rgb, 255:red, 0; green, 0; blue, 0 }  ][line width=0.75]    (10.93,-3.29) .. controls (6.95,-1.4) and (3.31,-0.3) .. (0,0) .. controls (3.31,0.3) and (6.95,1.4) .. (10.93,3.29)   ;
%Straight Lines [id:da9452835222880054] 
\draw [color={rgb, 255:red, 4; green, 4; blue, 253 }  ,draw opacity=1 ]   (368,112) -- (369,41) ;
%Straight Lines [id:da008847893143618224] 
\draw [color={rgb, 255:red, 4; green, 4; blue, 253 }  ,draw opacity=1 ]   (361,112) -- (361,39) ;
%Straight Lines [id:da5808070185930756] 
\draw [color={rgb, 255:red, 4; green, 4; blue, 253 }  ,draw opacity=1 ]   (401,113) -- (401,41) ;
%Straight Lines [id:da20053728191531528] 
\draw [color={rgb, 255:red, 4; green, 4; blue, 253 }  ,draw opacity=1 ]   (408,113) -- (408.53,64.99) -- (409,41) ;
%Straight Lines [id:da6855037252040743] 
\draw [color={rgb, 255:red, 4; green, 4; blue, 253 }  ,draw opacity=1 ]   (416,113) -- (417,42) ;
%Straight Lines [id:da6993079405198388] 
\draw [color={rgb, 255:red, 4; green, 4; blue, 253 }  ,draw opacity=1 ]   (507,112) -- (508,44) ;
%Straight Lines [id:da8921942674172503] 
\draw [color={rgb, 255:red, 4; green, 4; blue, 253 }  ,draw opacity=1 ]   (497,112) -- (498,43) ;
%Straight Lines [id:da673977492338834] 
\draw [color={rgb, 255:red, 4; green, 4; blue, 253 }  ,draw opacity=1 ]   (489,112) -- (488,42) ;
%Straight Lines [id:da8678122588566974] 
\draw [color={rgb, 255:red, 255; green, 0; blue, 0 }  ,draw opacity=1 ]   (351,112) -- (516,112) ;
%Straight Lines [id:da22717850827791808] 
\draw [color={rgb, 255:red, 4; green, 4; blue, 253 }  ,draw opacity=1 ]   (361,39) -- (350,18) ;
%Straight Lines [id:da118325733227852] 
\draw [color={rgb, 255:red, 4; green, 4; blue, 253 }  ,draw opacity=1 ]   (369,41) -- (350,18) ;
%Straight Lines [id:da47974780278850027] 
\draw [color={rgb, 255:red, 4; green, 4; blue, 253 }  ,draw opacity=1 ]   (401,41) -- (398,23) ;
%Straight Lines [id:da7846384655196452] 
\draw [color={rgb, 255:red, 4; green, 4; blue, 253 }  ,draw opacity=1 ]   (409,41) -- (398,23) ;
%Straight Lines [id:da6201757222077091] 
\draw [color={rgb, 255:red, 4; green, 4; blue, 253 }  ,draw opacity=1 ]   (417,42) -- (398,23) ;
%Straight Lines [id:da10778166061735517] 
\draw [color={rgb, 255:red, 4; green, 4; blue, 253 }  ,draw opacity=1 ]   (508,44) -- (519,23) ;
%Straight Lines [id:da8417463170554085] 
\draw [color={rgb, 255:red, 4; green, 4; blue, 253 }  ,draw opacity=1 ]   (498,43) -- (519,23) ;
%Straight Lines [id:da905794427861075] 
\draw [color={rgb, 255:red, 4; green, 4; blue, 253 }  ,draw opacity=1 ]   (488,42) -- (519,23) ;

% Text Node
\draw (232,101.4) node [anchor=north west][inner sep=0.75pt]    {$[ 0,1] \times T^{2}$};
% Text Node
\draw (122,147) node [anchor=north west][inner sep=0.75pt]   [align=left] {Step 1};
% Text Node
\draw (526,104.4) node [anchor=north west][inner sep=0.75pt]    {$[ 0,1] \times T^{2}$};
% Text Node
\draw (411,148) node [anchor=north west][inner sep=0.75pt]   [align=left] {Step 2};
% Text Node
\draw (9,25.4) node [anchor=north west][inner sep=0.75pt]    {$[ -1,1]$};
% Text Node
\draw (292,25.4) node [anchor=north west][inner sep=0.75pt]    {$[ -1,1]$};
% Text Node
\draw (50,-1.6) node [anchor=north west][inner sep=0.75pt]    {$e_{+}$};
% Text Node
\draw (213,0.4) node [anchor=north west][inner sep=0.75pt]    {$e_{-}$};
% Text Node
\draw (97,0.4) node [anchor=north west][inner sep=0.75pt]    {$\alpha _{i}$};
% Text Node
\draw (340,-3.6) node [anchor=north west][inner sep=0.75pt]    {$e_{+}$};
% Text Node
\draw (382,1.4) node [anchor=north west][inner sep=0.75pt]    {$\alpha _{i}$};
% Text Node
\draw (500,2.4) node [anchor=north west][inner sep=0.75pt]    {$e_{-}$};

\end{tikzpicture}
            \caption{Schematic representation of the surface $C$ and $S'$ in lemma \ref{explicitindex}. In this case $\alpha=\{(e_+,2),(\alpha_i,3),(e_-,2)\}$.}
            \label{TheSurface}
        \end{figure}
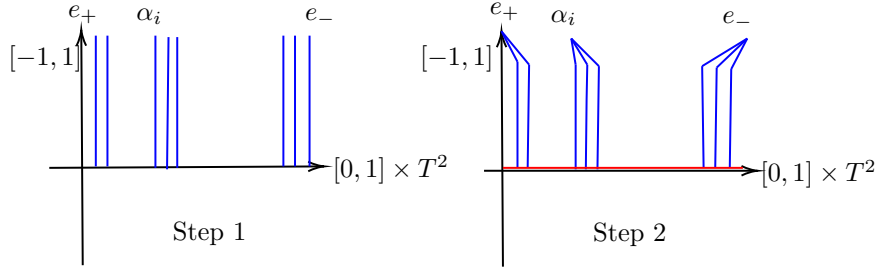

\begin{enumerate}
    \item (\textit{Relative Chern Class}) 
    We need a a generic section on $\xi|_S$ trivial with respect to $\tau$. Denote by $\eta$ the vector field over $S$ defined as $x(1-x)\partial_x$. It is easy to check that $\eta$ is a generic vector field trivial with respect to $\tau$. Notice that $\eta$ is $0$ exactly on the disks on $S$ obtained by the quotient that transform $S'$ into $S$. Therefore, $c_\tau(S)=\#\eta^{-1}(0)=c_1+c_2$.
    
    %Consider the quotient map $\pi:[0,1]\times \To^2\rightarrow L(n,m)$ defined by \ref{lensspace}. Notice that $\pi(\{0\}\times\To^2)$ is a disk invariant by the field $-n \partial_{t_1}+m \partial_{t_2}$. Anagously, $\pi(\{1\}\times\To^2)$ is a disk invariant by the field $\partial_{t_2}$. 
    \item (\textit{Relative self-intersection}) 
        To calculate $Q_\tau(S)$ we use the expression given in equation $\eqref{relativeQtau}$ 
            $$Q_\tau(S)=c_1(NS,\tau)+w_\tau(S)$$
        By construction $w_\tau(S)=0$. Let $\phi$ be the field obtained by projecting $\partial_s+\partial_x$ into $NS$. Notice that $\phi$ is $0$ exactly on the surgery points obtained by resolving singularities in step $2$. When resolving the singularites we obtained a number of zeros equal a determinant $\phi$ given by the resolution of the singularities. Since the determinant can be interpret as and area, a carefully organization of the terms will lead us to conclude that $\#\phi^{-1}(0)=2 A(\Lambda')$.
            
       \item (\textit{Conley-Zehnder Number}) Notice that equation \eqref{CZeq} follows directly from Lemma \ref{eq:CZelliptic}, the properties of a Morse-Bott perturbation explained in \ref{Morse-Bott} and \ref{lem:rotnumbers}. 
\end{enumerate}
\end{proof}

With this calculation in place we can conclude the proof of the part (b) of Proposition \ref{pro:equicomgeo}.

\begin{proof}[Proof of Proposition \ref{pro:equicomgeo}]
    
    Notice that from the above Proposition
    \begin{equation}
        I(\alpha)=2A(\Lambda)+c_1+c_2-e+m_++m_-+ 2\displaystyle\sum_{i=1}^{m_+}\left\lfloor -i\displaystyle\frac{a_2'(0)}{a'(0)\times (n,1)} \right\rfloor+2\displaystyle\sum_{i=1}^{m_-}\left\lfloor -j \displaystyle\frac{a_2'(1)}{a_1(1)} \right\rfloor
    \end{equation}
    
    We simplify this equation Using Pick's theorem on $Q_\tau$, which give us$$Q_\tau(\alpha)=2\iota(\Lambda_\alpha)+c_1+c_2+m_++m_-+e+h-1$$
    Where $\iota(\Lambda)$ is the count of the interior points of the closed region defined by $\Lambda$ and the axes without the boundary. Then 
    \begin{equation}
        \label{eq:comindexv1}
        \begin{split}
        I(\alpha)=&2\mathcal{L}(\Lambda)+2m_++2m_-+h
        \\+& 2\displaystyle\sum_{i=1}^{m_+}\left\lfloor -i\displaystyle\frac{a_2'(0)}{a'(0)\times (n,1)} \right\rfloor+2\displaystyle\sum_{i=1}^{m_-}\left\lfloor -j \displaystyle\frac{a_2'(1)}{a_1(1)} \right\rfloor
        \end{split}
    \end{equation}
    Now suppose that $\alpha$ does not contain exceptional orbits. This equation simplifies to 
    \begin{equation}
    \label{eq:comindexv2}
            I(\alpha)=2\mathcal{L}(\Lambda_\alpha)+h
    \end{equation}
    This in particular proves that $I(\alpha)=I(\Lambda)$ when $\alpha$ does not have exceptional orbits. Futhermore, this implies that the map $\alpha\mapsto\Lambda$ is a bijection between $\{\alpha:I(\alpha)=k\}$ and $\{\Lambda:I(\Lambda)=k\}$ for $L_0>0$ and $N$ big enough, in particular we take $N$ bigger than $k_0$. 
    
    We claim that the properties \textit{(a), (b), (c)} hold for the perturbation $\lambda^{L,N}$ with $L>L_0$.
    \begin{enumerate}[(a)]
        \item Notice that this follows from Lemma \ref{lem:therealperturbation}.
        \item From Equation \ref{eq:comindexv1} we notice that if $\alpha$ has an exceptional orbit $I(\alpha)>k_0$. Therefore, this item follows. 
        \item Now suppose that $\alpha$ does not contained exceptional orbits. We already argue that the map $\alpha\mapsto \Lambda$ is a bijection when we fixed $k$ under this perturbation. By Lemma \ref{lem:therealperturbation} the action of $\alpha$ and the $\Omega$-length of $\Lambda$ are $(1/L)$-closed.
    \end{enumerate}
    This finish the proof.
\end{proof}

%%%%%%%%%%%%%%%%%%%%%%%%%%%%%%%%%%%%%%%%%%%%%%%%%%%%%%%%%%%%%%%%%%%%%%%%%%%%%%%%%%%%%%%%%%%%%
\subsubsection{Embedded Contact Homology of Lens Spaces}
\label{ECH-of-Lens-Spaces}

Before continuing with the proof of Proposition \ref{pro:equicomgeo} we use the calculation of section \ref{ECHindexconcave} to compute the embedded contact homology of the lens spaces $L(n,1)$ when $\Gamma=0$. Here we use analogous arguments to the ones given in \cite[Sec.~3.7]{hutchings2014lecture} and in \cite[Theo.~7.6]{nelson2022embedded}. From these calculations we will deduce Lemma \ref{lem:capacitiesellipsoids} and with making use of Proposition \ref{pro:equicomgeo} we also prove Theorem \ref{thm:concavecapacities}.

\begin{lemm}[ECH index of the irrational ellipsoid with singularities] 
Suposse that $a$ and $b$ are numbers such that $a/b$ is irrational. We say that $X_\Omega=E_n(a,b)$ is an \textit{irrational ellipsoid with singularities}. The only Reeb orbits of the boundary $Y=\partial X_\Omega$ are the exceptional orbits $e^+$ and $e^-$ and  
\begin{align}
\label{eq:indexorbllipse}
I(e_+^re_-^s)&=nk^2+nk+2\displaystyle\sum_{i=1}^r\left(\left\lfloor i\displaystyle\frac{a-b}{nb} \right\rfloor +1\right)+2\displaystyle\sum_{i=1}^s\left(\left\lfloor i \displaystyle\frac{b-a}{na} \right\rfloor +1\right)    \\
&=nk(k+1)+2k+2\displaystyle\sum_{i=1}^r\left\lfloor i\displaystyle\frac{a-b}{nb} \right\rfloor+2\displaystyle\sum_{i=1}^s\left\lfloor i \displaystyle\frac{b-a}{na} \right\rfloor
\end{align}
Where $r+s=kn$.
\end{lemm}

\begin{proof}
 Notice that the claims about the elliptic orbits $e_+$ and $e_-$ follows from the comments at the beggining of Section \ref{sec:Reebdynamics}. Futhermore, we have that the rotations numbers are equal to $\phi_+=\displaystyle\frac{a-b}{nb}$ and $\phi_-=\displaystyle\frac{b-a}{na}$ by Lemma \ref{lem:rotnumbers}. 
By Lemma \ref{explicitindex} it follows that $Q_\tau(e_+^re_-^s)=k^2n$ and  $c_\tau(e_+^re_-^s)=nk$. This is enough to calculate the index and conclude the lemma.
\end{proof}
 
Now we prove the following central proposition.

\begin{pro}
\label{pro:bijectieven}
Let $Y=\partial E_{n}(a,b)$ such that $b/a$ is irrational. The $\text{ECH}$ index $I$ is a bijective map betweeen the Reeb orbits sets of $Y$ and the even numbers. 
\end{pro}

\begin{proof}
    For concretness, lets suppose that $b>a$. Now lets consider $(x,y)$ a pair of lattice points in $V_n$. Lets call $L$ the line with slope $b/a$ that passes through $(x,y)$. Let $P$ be the number of lattice points contained in the enclosed region defined by $L$ and the axis without counting the point $(x,y)$. Futhermore, consider $s$ the horizontal distance from the $y$-axis to the point $(x,y)$, lets called the segment defined $S$. Similarly, lets call $r$ the horizontal distance from $(x,y)$ to the $(n,1)$-axis and lets call this segment $R$. Notice that it is enough to check that $I(e_+^re_-^s)=2P$.

    To check this, we need some setting, see Figure \ref{fig:counting} as a guide. Lets call $A$ the region defined by the segment of $L$ above $S$, the segment $S$ itself and the $y$-axis. By $B$ the region defined by the axis, the segment $S$ and by the segment of $L$ below $S$. Finally, lets denote by $C$ the region defined by the segment of $L$ below $R$, the segment $R$ itself and the $(n,1)$-axis. 

    We now analize carefully the Equation \ref{eq:indexorbllipse} to prove that it counts the lattice points correctly. By Pick's Theorem $nk^2=2\iota(B\cup  C)+2k+r+s$ where $\iota(B\cup A )$ is the number of interior points of $B\cup A$. The term $a=\sum_{i=1}^s\left\lfloor i \frac{b-a}{na} \right\rfloor$ counts the lattice points in $A$ without counting the lattice points in $S$. The term $-c=\sum_{i=1}^s\left\lfloor i \frac{b-a}{na} \right\rfloor$ counts (negatively) the lattice points in $C$ including the lattice points in $R$ without including the $(x,y)$ point. Summing these quantities we obtain that 
    \begin{align*}
        I(e_+^re^s_-)&=2\iota(B\cup  C)+2k+2r+2s+2a-2c\\
        &=2P +2\iota(C) +2k'+2r-2c
    \end{align*}
    Where $k'$ is the number of lattice point that lies in the $(n,1)$-axis and the region $C$. However, $c=\iota(C)+k'+r$. Which finish the proof. 
\end{proof}

	\begin{figure}
		\centering
		\tikzset{every picture/.style={line width=0.75pt}} %set default line width to 0.75pt        

\begin{tikzpicture}[x=0.75pt,y=0.75pt,yscale=-1,xscale=1]

%Shape: Grid [id:dp910708938291662] 
\draw  [draw opacity=0] (10,8) -- (459.67,8) -- (459.67,292) -- (10,292) -- cycle ; \draw  [color={rgb, 255:red, 0; green, 0; blue, 0 }  ,draw opacity=0.33 ] (10,8) -- (10,292)(30,8) -- (30,292)(50,8) -- (50,292)(70,8) -- (70,292)(90,8) -- (90,292)(110,8) -- (110,292)(130,8) -- (130,292)(150,8) -- (150,292)(170,8) -- (170,292)(190,8) -- (190,292)(210,8) -- (210,292)(230,8) -- (230,292)(250,8) -- (250,292)(270,8) -- (270,292)(290,8) -- (290,292)(310,8) -- (310,292)(330,8) -- (330,292)(350,8) -- (350,292)(370,8) -- (370,292)(390,8) -- (390,292)(410,8) -- (410,292)(430,8) -- (430,292)(450,8) -- (450,292) ; \draw  [color={rgb, 255:red, 0; green, 0; blue, 0 }  ,draw opacity=0.33 ] (10,8) -- (459.67,8)(10,28) -- (459.67,28)(10,48) -- (459.67,48)(10,68) -- (459.67,68)(10,88) -- (459.67,88)(10,108) -- (459.67,108)(10,128) -- (459.67,128)(10,148) -- (459.67,148)(10,168) -- (459.67,168)(10,188) -- (459.67,188)(10,208) -- (459.67,208)(10,228) -- (459.67,228)(10,248) -- (459.67,248)(10,268) -- (459.67,268)(10,288) -- (459.67,288) ; \draw  [color={rgb, 255:red, 0; green, 0; blue, 0 }  ,draw opacity=0.33 ]  ;
%Straight Lines [id:da49733355597886475] 
\draw [color={rgb, 255:red, 8; green, 8; blue, 8 }  ,draw opacity=1 ]   (50,268) -- (50,30) ;
\draw [shift={(50,28)}, rotate = 90] [color={rgb, 255:red, 8; green, 8; blue, 8 }  ,draw opacity=1 ][line width=0.75]    (10.93,-3.29) .. controls (6.95,-1.4) and (3.31,-0.3) .. (0,0) .. controls (3.31,0.3) and (6.95,1.4) .. (10.93,3.29)   ;
%Straight Lines [id:da26278001188115574] 
\draw [color={rgb, 255:red, 0; green, 0; blue, 0 }  ,draw opacity=1 ]   (22.67,260) -- (448.18,68.82) ;
\draw [shift={(450,68)}, rotate = 155.81] [color={rgb, 255:red, 0; green, 0; blue, 0 }  ,draw opacity=1 ][line width=0.75]    (10.93,-3.29) .. controls (6.95,-1.4) and (3.31,-0.3) .. (0,0) .. controls (3.31,0.3) and (6.95,1.4) .. (10.93,3.29)   ;
%Straight Lines [id:da26294892105680323] 
\draw [color={rgb, 255:red, 0; green, 6; blue, 255 }  ,draw opacity=1 ]   (19.67,54) -- (261.67,260) ;
%Straight Lines [id:da814302453971934] 
\draw [color={rgb, 255:red, 255; green, 0; blue, 0 }  ,draw opacity=1 ]   (270,148) -- (252,148) ;
\draw [shift={(250,148)}, rotate = 360] [color={rgb, 255:red, 255; green, 0; blue, 0 }  ,draw opacity=1 ][line width=0.75]    (10.93,-3.29) .. controls (6.95,-1.4) and (3.31,-0.3) .. (0,0) .. controls (3.31,0.3) and (6.95,1.4) .. (10.93,3.29)   ;
%Straight Lines [id:da8614830598658886] 
\draw [color={rgb, 255:red, 255; green, 0; blue, 0 }  ,draw opacity=1 ]   (70,148) -- (52,148) ;
\draw [shift={(50,148)}, rotate = 360] [color={rgb, 255:red, 255; green, 0; blue, 0 }  ,draw opacity=1 ][line width=0.75]    (10.93,-3.29) .. controls (6.95,-1.4) and (3.31,-0.3) .. (0,0) .. controls (3.31,0.3) and (6.95,1.4) .. (10.93,3.29)   ;
%Straight Lines [id:da2107478795871145] 
\draw [color={rgb, 255:red, 255; green, 0; blue, 0 }  ,draw opacity=1 ]   (90,148) -- (72,148) ;
\draw [shift={(70,148)}, rotate = 360] [color={rgb, 255:red, 255; green, 0; blue, 0 }  ,draw opacity=1 ][line width=0.75]    (10.93,-3.29) .. controls (6.95,-1.4) and (3.31,-0.3) .. (0,0) .. controls (3.31,0.3) and (6.95,1.4) .. (10.93,3.29)   ;
%Straight Lines [id:da9045845715642956] 
\draw [color={rgb, 255:red, 255; green, 0; blue, 0 }  ,draw opacity=1 ]   (110,148) -- (92,148) ;
\draw [shift={(90,148)}, rotate = 360] [color={rgb, 255:red, 255; green, 0; blue, 0 }  ,draw opacity=1 ][line width=0.75]    (10.93,-3.29) .. controls (6.95,-1.4) and (3.31,-0.3) .. (0,0) .. controls (3.31,0.3) and (6.95,1.4) .. (10.93,3.29)   ;
%Straight Lines [id:da7292480917655011] 
\draw [color={rgb, 255:red, 255; green, 0; blue, 0 }  ,draw opacity=1 ]   (130,148) -- (112,148) ;
\draw [shift={(110,148)}, rotate = 360] [color={rgb, 255:red, 255; green, 0; blue, 0 }  ,draw opacity=1 ][line width=0.75]    (10.93,-3.29) .. controls (6.95,-1.4) and (3.31,-0.3) .. (0,0) .. controls (3.31,0.3) and (6.95,1.4) .. (10.93,3.29)   ;
%Straight Lines [id:da09267147176905599] 
\draw [color={rgb, 255:red, 255; green, 0; blue, 0 }  ,draw opacity=1 ]   (150,148) -- (132,148) ;
\draw [shift={(130,148)}, rotate = 360] [color={rgb, 255:red, 255; green, 0; blue, 0 }  ,draw opacity=1 ][line width=0.75]    (10.93,-3.29) .. controls (6.95,-1.4) and (3.31,-0.3) .. (0,0) .. controls (3.31,0.3) and (6.95,1.4) .. (10.93,3.29)   ;
%Straight Lines [id:da46219085665980697] 
\draw [color={rgb, 255:red, 255; green, 0; blue, 0 }  ,draw opacity=1 ]   (170,148) -- (152,148) ;
\draw [shift={(150,148)}, rotate = 360] [color={rgb, 255:red, 255; green, 0; blue, 0 }  ,draw opacity=1 ][line width=0.75]    (10.93,-3.29) .. controls (6.95,-1.4) and (3.31,-0.3) .. (0,0) .. controls (3.31,0.3) and (6.95,1.4) .. (10.93,3.29)   ;
%Straight Lines [id:da5140949149648218] 
\draw [color={rgb, 255:red, 255; green, 0; blue, 0 }  ,draw opacity=1 ]   (190,148) -- (172,148) ;
\draw [shift={(170,148)}, rotate = 360] [color={rgb, 255:red, 255; green, 0; blue, 0 }  ,draw opacity=1 ][line width=0.75]    (10.93,-3.29) .. controls (6.95,-1.4) and (3.31,-0.3) .. (0,0) .. controls (3.31,0.3) and (6.95,1.4) .. (10.93,3.29)   ;
%Straight Lines [id:da6080186153055223] 
\draw [color={rgb, 255:red, 255; green, 0; blue, 0 }  ,draw opacity=1 ]   (210,148) -- (192,148) ;
\draw [shift={(190,148)}, rotate = 360] [color={rgb, 255:red, 255; green, 0; blue, 0 }  ,draw opacity=1 ][line width=0.75]    (10.93,-3.29) .. controls (6.95,-1.4) and (3.31,-0.3) .. (0,0) .. controls (3.31,0.3) and (6.95,1.4) .. (10.93,3.29)   ;
%Straight Lines [id:da9623426807413761] 
\draw [color={rgb, 255:red, 255; green, 0; blue, 0 }  ,draw opacity=1 ]   (230,148) -- (212,148) ;
\draw [shift={(210,148)}, rotate = 360] [color={rgb, 255:red, 255; green, 0; blue, 0 }  ,draw opacity=1 ][line width=0.75]    (10.93,-3.29) .. controls (6.95,-1.4) and (3.31,-0.3) .. (0,0) .. controls (3.31,0.3) and (6.95,1.4) .. (10.93,3.29)   ;
%Straight Lines [id:da7055308776166791] 
\draw [color={rgb, 255:red, 255; green, 0; blue, 0 }  ,draw opacity=1 ]   (250,148) -- (232,148) ;
\draw [shift={(230,148)}, rotate = 360] [color={rgb, 255:red, 255; green, 0; blue, 0 }  ,draw opacity=1 ][line width=0.75]    (10.93,-3.29) .. controls (6.95,-1.4) and (3.31,-0.3) .. (0,0) .. controls (3.31,0.3) and (6.95,1.4) .. (10.93,3.29)   ;

% Text Node
\draw (66,118) node [anchor=north west][inner sep=0.75pt]   [align=left] {A};
% Text Node
\draw (86,180) node [anchor=north west][inner sep=0.75pt]   [align=left] {B};
% Text Node
\draw (174,159) node [anchor=north west][inner sep=0.75pt]   [align=left] {C};
% Text Node
\draw (226,213) node [anchor=north west][inner sep=0.75pt]   [align=left] {L};

\end{tikzpicture}    
		\caption{Example for $L(2,1)$ of Proposition \ref{pro:bijectieven}.}
		\label{fig:counting}
	\end{figure}
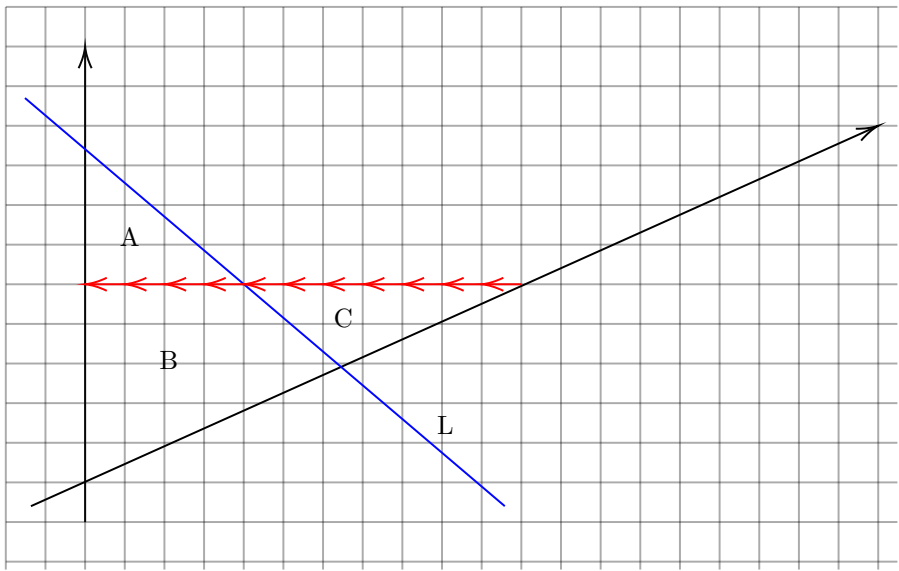

\begin{col} Suppose that $Y$ is a lens space. Then
\label{ECHlenscol}
\begin{align}
\label{ECHlens}
\text{ECH}_*(Y,\emptyset) = \left\{ 
    \begin{array}{lcc}
                \mathbb{Z}_2     &   \text{if } &  *=2k \\
                       0         &   \text{if } &  *=2k+1
             \end{array}
   \right.
\end{align}    
\end{col}

\begin{proof}
This follows directly from the fact that the differential is an operator of index $-1$ and  Proposition \ref{pro:bijectieven} it follows that the differential is equal to zero and therefore the homology correspond to the one given in equation \ref{ECHlens}. The result follows from the topological invariance of the embedded contact homology. 
\end{proof}

Now we can prove Lemma \ref{lem:capacitiesellipsoids} where we stablish the capacities for the ellipsoids with singularity.

\begin{proof}[proof of Lemma \ref{lem:capacitiesellipsoids}]
    Lets suppose first that $a/b$ is irrational. Notice that  $\mathcal{A}(e_1^{r}e_2^{s})=ra+sb$ then it follows from Corollary \ref{ECHlenscol} that the capacities are indeed a reorganization of the numbers $k_1a+k_2b$ with the additional condition that $r+s=kn$ for some $k$. This stablish the result follows for $a/b$ irrational, by continuity of the ECH capacities we get the general case. 
\end{proof}

\subsubsection{ECH spectrum of Concave toric Lens spaces.}
\label{sec:ECHspectrum}

In the most important cases for our purpouses we can use a simpler version of the $\text{ECH}$ spectrum. We say that a closed sum of generators $\alpha_1+\cdots+\alpha_r$ is \textit{minimal} if after removing any number of summands the sum is no longer closed. 

\begin{lemm}
\label{lem:capsimp}
Let $(Y,\lambda)$ be a contact closed 3-manifold. Suppose that 
\begin{align*}
\text{ECH}_*(Y,\lambda,0) = \left\{ 
    \begin{array}{lcc}
                \mathbb{Z}_2     &   \text{if}  & *=2k \\
                       0         &   \text{if}  & *=2k+1
             \end{array}
   \right.
\end{align*}
and that the $U$ map of $\text{ECH}(Y,\lambda)$ is an isomorphism for every even index. Then
\begin{equation}
    \label{specsimp}
    \begin{split}
    c_k(Y,\lambda)=\min\{\max\{\mathcal{A}(\alpha_1),\cdots,\mathcal{A}(\alpha_r)\}: 
    I(\alpha_1)=\cdots=I(\alpha_r)=2k,  \\
    \alpha_1+\cdots+\alpha_r \text{ is minimal and non-nullhomologous}\}
    \end{split}
\end{equation}
\end{lemm}

\begin{proof} Since $U$ is an isomorphism and $ECH_{2k}=\mathbb{Z}_2$, the expression for the capacities simplifies as 
\begin{align*}
        c_k(Y,\lambda)=\inf\{L:\eta\in \text{ECH}_{2k}^L(Y,\lambda), \eta\not=0\}
    \end{align*}
Take $L>0$, suppose that $\eta\in \text{ECH}_{2k}^L(Y,\lambda)$ and $\eta\not=0$.

Suppose that $\eta=[\alpha_1+\cdots +\alpha_r]$ then by definition $\max\{\mathcal{A}(\alpha_1),\dots,\mathcal{A}(\alpha_r)\}<L$. It follows that 
$$\min\{\max\{\mathcal{A}(\alpha_1),\dots,\mathcal{A}(\alpha_r)\}:\eta=[\alpha_1+\cdots +\alpha_r]\}\leq c_k(Y,\lambda)$$

Take $L$ equal to the left side of the above inequality then for every $\epsilon>0$ there exist a sum of generators $\alpha_1+\cdots+\alpha_r$ such that $[\alpha_1+\cdots+\alpha_r]=\eta\in \text{ECH}^{L+\epsilon}(Y,
\lambda)$ then $c_k(Y,\lambda)\leq L$. The result follows. 
\end{proof}

By Corollary \ref{ECHlenscol} the Equation (\ref{specsimp}) holds for any lens space. Finally, we can give the definition of the $\text{ECH}$ capacities.

\begin{lemm}
\label{lem:capupperbound}
Suppose that $X_\Omega$ is a concave toric domain in $M_n$ then
    \begin{equation}
    c_k(X_\Omega)\leq \max \{l_\Omega(\Lambda):\mathcal{L}_n(\Lambda)=k\} 
    \end{equation}
    Where $\Lambda$ run over all elliptic concave generators. 
\end{lemm}

\begin{proof}Fix $k$. Choose $L>0$ sufficiently big so that $X^{L}$ is an $L$-flat perturbation of $X_\Omega$ with respect to $k$ as guaranteed by Proposition \ref{pro:equicomgeo}. Notice that 
$$c_k(X^{L})\leq \max\{\mathcal{A}(\alpha):I(\alpha)=2k\}$$
By Lemma \ref{lem:capsimp}.  By Proposition \ref{pro:equicomgeo}(c) we have that 
$$\max\{\mathcal{A}(\alpha):I(\alpha)=2k\}\leq \max\{l_\Omega(\Lambda):I(\Lambda)=2k\}+\frac{1}{L}$$

Now we want to argue that $\max\{l_\Omega(\Lambda):I(\Lambda)=2k\}\leq \max\{l_\Omega(\Lambda):\mathcal{L}(\Lambda)=k\}$ where the right hand side runs over elliptic concave generators. Take $\Lambda$ with $I(\Lambda)=2k$ and suppose that $\Lambda$ has an hyperbolic labels. By index considerations $\Lambda$ has at least two hyperbolic labels. Now by Equation \ref{eq:echinccon} changing one label `h' by an label `e' on $\Lambda$ give us a concave generator $\Lambda'$ with $I(\Lambda')=2k-1$ with at least one label `h'. By doing a rounding the corner on $\Lambda'$ and removing the remaining `h' we obtain a generator $\Lambda''$ with $I(\Lambda'')=2k$. By Corollary  \ref{col:cor-red-action} we have that $l_\Omega(\Lambda)\leq l_\Omega(\Lambda'')$. Therefore, the maximum is contained in the set of all the elliptice concave generators $\Lambda$, as we claimed.  

Therefore, 
$$c_k(X^{L})\leq \max\{l_\Omega(\Lambda):\mathcal{L}(\Lambda)=k\}+\frac{1}{L}$$
Where $\Lambda$ runs over all elliptic concave generators. By taking $L$ going to infinity we obtain the result. 
    
\end{proof}

\begin{proof}[Proof of Theorem \ref{thm:blowupconcavecapacities}] The proof of this version of the theorem can be carried out in an analogous as the steps explained in Section \ref{sec:restofthepaper} with $l_\Omega^\delta$ instead of $l_\Omega$ and some minor changes. This is due to the fact that the term $y(\Lambda)$ in the definition of $l_\Omega^\delta$ is linear. In the definitino of the weight expansion in Section \ref{sec:com<packing} the first term $w_0$ correspond to a singular ball $B(w_0)$ after a rational blow up of size $\delta$ is done on it. 
\end{proof}

%%%%%%%%%%%%%%%%%%%%%%%%%%%%%%%%%%%%%%%%%%%%%%%%%%%%%%%%%%%%%%%%%%%%%%%%%%%%%%%%%%%%%%%%%%%%%%%%%%%%%

\printbibliography

\end{document}